\documentclass[10pt]{article}
\usepackage{amsmath}                        
\usepackage{amsthm}                         
\usepackage{amssymb}                        
\usepackage{amsfonts}                       
\usepackage{latexsym}                       
\usepackage{setspace}
\usepackage{appendix}	
\usepackage[colorlinks=true,linkcolor=blue,citecolor=blue]{hyperref}
\usepackage[all]{xy}

\newtheorem{thm}{Theorem}

\newtheorem{lem}{Lemma}[section]
\newtheorem{prop}[lem]{Proposition}
\newtheorem{coro}{Corollary}
\newtheorem{remar}{Remark}

\newcommand{\R}{\mathbb{R}}
\newcommand{\N}{\mathbb{N}}
\newcommand{\C}{\mathbb{C}}
\newcommand{\A}{\mathcal{A}}
\newcommand{\Q}{\mathcal{Q}}
\newcommand{\U}{\mathcal{U}}
\newcommand{\AC}{\mathcal{A}_{\mathbb{C}}}

\newcommand{\HAC}{\mathbb{H}^{\mathcal{A}_{\mathbb{C}}}}
\newcommand{\e}{\textit{1}}

\newcommand{\la}{\langle}
\newcommand{\ra}{\rangle}
\newcommand{\laa}{\langle\langle}
\newcommand{\raa}{\rangle\rangle}
\title{\bf Spinor representation of Lorentzian surfaces in $\R^{2,2}$}

\author{Pierre Bayard\footnote{bayard@ciencias.unam.mx, Facultad de Ciencias, Universidad Nacional Aut\'onoma de M\'exico, M\'exico}, Victor Patty\footnote{victorp@ifm.umich.mx, Instituto de F\'{\i}sica y Matem\'aticas, U.M.S.N.H., Ciudad Universitaria, CP. 58040 Morelia, Michoac\'an, M\'exico}}

\date{}

\begin{document}
\maketitle

\begin{abstract}
We prove that an isometric immersion of a simply connected Lorentzian surface in $\R^{2,2}$ is equivalent to a normalised spinor field solution of a Dirac equation on the surface. Using the quaternions and the Lorentz numbers, we also obtain an explicit representation formula of the immersion in terms of the spinor field. We then apply the representation formula in $\R^{2,2}$ to give a new spinor representation formula for Lorentzian surfaces in 3-dimensional Minkowski space. Finally, we apply the representation formula to the local description of the flat Lorentzian surfaces with flat normal bundle and regular Gauss map in $\R^{2,2},$ and show that these surfaces locally depend on four real functions of one real variable, or on one holomorphic function together with two real functions of one real variable, depending on the sign of a natural invariant.
\end{abstract}
\maketitle\pagenumbering{arabic}

\section*{Introduction}
Let $\R^{2,2}$ be the space $\R^4$ endowed with the metric of signature $(2,2)$ 
$$g=-dx_0^2+dx_1^2-dx_2^2+dx_3^2.$$
A surface $M\subset\R^{2,2}$ is said to be Lorentzian if the metric $g$ induces on $M$ a Lorentzian metric, i.e. a metric  of signature $(1,1):$ the tangent and the normal bundles of $M$ are then equipped with fibre Lorentzian metrics. The purpose of the paper is to study the spinor representation of the Lorentzian surfaces in $\R^{2,2};$ the main result is the following: if $M$ is an abstract Lorentzian surface, $E$ is a bundle of rank 2 on $M,$ with a Lorentzian fibre metric and a compatible connection, and $\vec{H}\in\Gamma(E)$ is a section of $E,$ then an isometric immersion of $M$ into $\R^{2,2},$ with normal bundle $E$ and mean curvature vector $\vec{H},$ is equivalent to a normalised section $\varphi\in\Gamma(\Sigma),$ solution of a Dirac equation $D\varphi=\vec{H}\cdot \varphi$ on the surface, where $\Sigma=\Sigma E\otimes \Sigma M$ is the spinor bundle of $M$ twisted by the spinor bundle of $E$ and $D$ is a natural Dirac operator acting on $\Sigma$ (we assume that spin structures are given on $TM$ and $E$). We moreover define a natural closed 1-form $\xi$ in terms of $\varphi,$ with values in $\R^{2,2},$ such that $F:=\int\xi$ is the immersion. As a first application of this representation, we derive an easy proof of the fundamental theorem of the theory of Lorentzian surfaces immersed in $\R^{2,2}:$ a symmetric bilinear map $B:TM\times TM\rightarrow E$ is the second fundamental form of an immersion of $M$ into $\R^{2,2}$ if and only if it satisfies the equations of Gauss, Codazzi and Ricci. We then deduce from the general representation in $\R^{2,2}$ spinor representations for Lorentzian surfaces in 3-dimensional Minkowski spaces $\R^{1,2}$ and $\R^{2,1},$ and also obtain new explicit representation formulas; the representations appear to be simpler than the representations obtained before by M.-A. Lawn \cite{lawn_thesis,lawn} and by M.-A. Lawn and J. Roth \cite{lawn_roth}, since only one spinor field is involved in the formulas. Our last application concerns the flat Lorentzian surfaces with flat normal bundle and regular Gauss map in $\R^{2,2}$: the general spinor representation formula permits to study their local structure; they locally depend on four real functions of one real variable if a natural invariant $\Delta$ is positive, and on one holomorphic function together with two real functions of one real variable if $\Delta$ is negative.

We note that a spinor representation for surfaces in 4-dimensional pseudo-Riemannian spaces already appeared in \cite{Va}; the representation formula obtained in that paper seems to be different, since the normal bundle and the Clifford action are not explicitly involved in the formula.

We quote the following related papers: the spinor representation of surfaces in $\R^3$ was studied by many authors, especially by Th. Friedrich in \cite{friedrich}, who interpreted a spinor field representing a surface in $\R^3$ as a constant spinor field of $\R^3$ restricted to the surface;  following this approach, the spinor representation of Lorentzian surfaces in 3-dimensional Minkowski space was studied by M.-A. Lawn \cite{lawn_thesis,lawn} and M.-A. Lawn and J. Roth \cite{lawn_roth}. M.-A. Lawn, J. Roth and the first author then studied the spinor representation of surfaces in 4-dimensional Riemannian space forms in \cite{bayard_lawn_roth}, and the first author the spinor representation of spacelike surfaces in 4-dimensional Minkowski space in \cite{bayard1}. Recently, P. Romon and J. Roth studied in \cite{RR} the relation between this abstract approach and more explicit representation formulas existing in the literature for surfaces in $\R^3$ and $\R^4$. Finally, the local description of the flat surfaces with flat normal bundle and regular Gauss map in 4-dimensional Euclidean and Minkowski spaces was studied in \cite{dajczer tojeiro}.

The outline of the paper is as follows: the first section is devoted to preliminaries concerning the Clifford algebra of $\R^{2,2},$ the spin representation, and the spin geometry of Lorentzian surfaces in $\R^{2,2}.$ We use quaternions and Lorentz numbers to obtain concise formulas. Section \ref{section spin representation} is devoted to the spinor representation formula of Lorentzian surfaces in $\R^{2,2}.$ We indicate at the end of the section how to obtain the representation formulas for surfaces in $\R^{1,2}$ and $\R^{2,1}.$ We then apply the representation formula to the local description of the flat Lorentzian surfaces with flat normal bundle and regular Gauss map in Section \ref{section flat surfaces}. An appendix ends the paper.
\section{Preliminaries}\label{prelim}
\subsection{Clifford algebra of $\R^{2,2}$ and the spin representation}\label{prelim subsection1}
Let us denote by $(e_0,e_1,e_2,e_3)$ the canonical basis of $\R^{2,2}.$ The norm of a vector $x=(x_0,x_1,x_2,x_3)$ belonging to $\R^{2,2}$ is
$$\langle x,x\rangle:=-x_0^2+x_1^2-x_2^2+x_3^2.$$ 
To describe the Clifford algebra of $\R^{2,2},$ it will be convenient to consider the Lorentz numbers 
$$\A=\{u+\sigma v:\  u,v\in\R\},$$
where $\sigma$ is a formal element such that $\sigma^2=1,$ the complexified Lorentz numbers
$$\AC:=\mathcal{A}\otimes\mathbb{C}\simeq\{u+\sigma v:\  u,v\in\C\},$$ 
and the quaternions  with coefficients in $\AC$ 
$$\HAC:=\{ \zeta_0\e+\zeta_1I+\zeta_2J+\zeta_3K:\ \zeta_0,\zeta_1,\zeta_2,\zeta_3 \in \AC\},$$ 
where $I,J$ and $K$ are such that 
$$I^2=J^2=K^2=-\e,\hspace{.5cm}IJ=-JI=K.$$ 
If $a=u+\sigma v$ belongs to $\AC,$ we denote $\widehat{a}:=u-\sigma v,$ and set, for all $\zeta=\zeta_0\e+\zeta_1I+\zeta_2J+\zeta_3K$ belonging to $\HAC,$ 
$$\widehat{\zeta}:=\widehat{\zeta_0}\e+\widehat{\zeta_1}I+\widehat{\zeta_2}J+\widehat{\zeta_3}K.$$ 
If $\HAC(2)$ stands for the set of $2\times2$ matrices with entries belonging to $\HAC,$ the map
\begin{eqnarray}
\gamma:\hspace{.5cm}\R^{2,2} & \longrightarrow & \HAC(2)\label{aplicliff} \\
(x_0,x_1,x_2,x_3) & \longmapsto & 
\begin{pmatrix} 0 & \sigma i x_0\e+x_1I+ix_2J+x_3K\\ -\sigma i x_0\e+x_1I+ix_2J+x_3K & 0 \end{pmatrix}\nonumber
\end{eqnarray}
is a Clifford map, that is satisfies
$$\gamma(x)^2=-\langle x,x\rangle\left(\begin{array}{cc}\e&0\\0&\e\end{array}\right)$$ 
for all $x\in\R^{2,2},$ and thus identifies
\begin{equation}\label{description Cl22} 
Cl(2,2)\simeq\left\lbrace\begin{pmatrix} p & q\\ \widehat{q} & \widehat{p} \end{pmatrix}:\ p\in \mathbb{H}_0,\ q\in \mathbb{H}_1 \right\rbrace, 
\end{equation}
where 
$$\mathbb{H}_0:=\left\lbrace p_0\e+ip_1I+p_2J+ip_3K:\ p_0,p_1,p_2,p_3\in\A\right\rbrace$$
and
$$\mathbb{H}_1:=\left\lbrace iq_0\e+q_1I+iq_2J+q_3K:\ q_0,q_1,q_2,q_3\in\A\right\rbrace.$$
Note that $\mathbb{H}_0$ naturally identifies to the para-quaternions numbers described in \cite{lawn_thesis}, but here with coefficients in the Lorentz numbers $\A.$ Using (\ref{description Cl22}), the sub-algebra of elements of even degree is   
\begin{align}\label{elempar}
Cl_0(2,2)& \simeq\left\lbrace \begin{pmatrix}
p & 0\\
0 & \widehat{p}
\end{pmatrix}:\ p\in\mathbb{H}_0\right\rbrace\simeq\mathbb{H}_0
\end{align}
and the set of elements of odd degree is
\begin{align}\label{elem_impar}
Cl_1(2,2)& \simeq\left\lbrace \begin{pmatrix}
0 & q\\
\widehat{q} & 0
\end{pmatrix}:\ q\in\mathbb{H}_1\right\rbrace\simeq\mathbb{H}_1.
\end{align}
If $\zeta=\zeta_0\e+\zeta_1I+\zeta_2J+\zeta_3K$ belongs to $\HAC,$ we define its conjugate by 
$$\overline{\zeta}:=\zeta_0\e-\zeta_1I-\zeta_2J-\zeta_3K.$$ 
Let us consider the map 
\begin{eqnarray*}
H:\hspace{1cm}\HAC\times\HAC &\longrightarrow& \AC \\
(\zeta,\zeta') &\longmapsto& \frac{1}{2}\left(\zeta\overline{\zeta'}+\zeta'\overline{\zeta}\right)
=\zeta_0\zeta_0'+\zeta_1\zeta_1'+\zeta_2\zeta_2'+\zeta_3\zeta_3'
\end{eqnarray*}
where $\zeta=\zeta_0\e+\zeta_1I+\zeta_2J+\zeta_3K$ and $\zeta'=\zeta_0'\e+\zeta_1'I+\zeta_2'J+\zeta_3'K.$ It is obviously $\AC$-bilinear and symmetric. If we consider the restriction of this map to $\mathbb{H}_0,$ 
\begin{equation}\label{aplic_H_restricta} 
H(p,p')=p_0p_0'-p_1p_1'+p_2p_2'-p_3p_3'\hspace{.5cm}\in\ \A
\end{equation}
where $p=p_0\e+ip_1I+p_2J+ip_3K$ and $p'=p_0'\e+ip_1'I+p_2'J+ip_3'K$ belong to $\mathbb{H}_0,$ the spin group is given by
\begin{equation*}
Spin(2,2):=\left\lbrace p\in \mathbb{H}_0:\ H(p,p)=1 \right\rbrace\subset Cl_0(2,2).
\end{equation*}
Now, if we consider the identification 
\begin{align}\label{iden_espa}
\R^{2,2} & \simeq \{\sigma i x_0\e+x_1I+ix_2J+x_3K:\ x_0,x_1,x_2,x_3\in \R\} \notag \\
& \simeq \{q \in \mathbb{H}_1:\ q=-\widehat{\overline{q}}\}, 
\end{align}
we get the double cover  
\begin{eqnarray}\label{cubriente}
\Phi: & Spin(2,2) & \longrightarrow  SO(2,2)\\
& p & \longmapsto  (q\in\R^{2,2} \longmapsto p q\widehat{p}^{-1}\in\R^{2,2}). \notag
\end{eqnarray} 
Here and below $SO(2,2)$ stands for the component of the identity of the orthogonal group $O(2,2)$ (elementary properties of this group may be found in \cite{oneill}). 

If we consider $\mathbb{H}_0$ as a complex vector space, with the complex structure given by the multiplication by $J$ on the right, the complex irreducible representation of $Cl(2,2)$ can be conveniently represented as follows: 
$$\rho: Cl(2,2) \longrightarrow End(\mathbb{H}_0)$$ 
where
\begin{equation}\label{rep_alg} \rho\begin{pmatrix}p & q\\ \widehat{q} & \widehat{p}\end{pmatrix}:\hspace{0.2in} \xi\in\mathbb{H}_0\hspace{.2cm}\simeq \hspace{.2cm}\begin{pmatrix}\xi \\ \sigma i\widehat{\xi}\end{pmatrix}\hspace{.3cm}\longmapsto\hspace{.3cm} \begin{pmatrix}p & q\\ \widehat{q} & \widehat{p}\end{pmatrix}\begin{pmatrix}\xi \\ \sigma i\widehat{\xi}\end{pmatrix}\hspace{.2cm}\simeq\hspace{.2cm} p\xi+\sigma iq\widehat{\xi}\in\mathbb{H}_0,
\end{equation} 
so that the spinorial representation of $Spin(2,2)$ simply reads 
\begin{eqnarray}\label{rep_spin}
\rho_{|Spin(2,2)}: Spin(2,2) & \longrightarrow & End_{\C}(\mathbb{H}_0) \\
p & \longmapsto & (\xi\in \mathbb{H}_0 \longmapsto p\xi\in\mathbb{H}_0). \notag
\end{eqnarray}
Since $\rho(\sigma \e)^2=id_{\mathbb{H}_0},$ this representation splits into  
\begin{equation}\label{splitting H0}
\mathbb{H}_0=\Sigma^+\oplus \Sigma^-,
\end{equation} 
where $\Sigma^+:=\{\xi\in \mathbb{H}_0:\ \sigma\xi=\xi\}$ and $\Sigma^-:=\{\xi\in \mathbb{H}_0:\ \sigma\xi=-\xi\}.$ Explicitly, we have 
\begin{equation}\label{exp_expli_pos}
\Sigma^+ =(1+\sigma)\left\{(\R\oplus\R J)+iI(\R\oplus\R J)\right\}
\end{equation}
and
\begin{equation}\label{exp_expli_neg}
\Sigma^- =(1-\sigma)\left\{(\R\oplus\R J)+iI(\R\oplus\R J)\right\}.
\end{equation}
Note that $\sigma\e\in\mathbb{H}_0$ represents the volume element $e_0\cdot e_1\cdot e_2\cdot e_3,$ which thus acts as $+id$ on $\Sigma^+$ and as $-id$ on $\Sigma^-.$

\subsection{Spinors under the splitting $\R^{2,2}=\R^{1,1}\times \R^{1,1}$}
We consider the splitting $\R^{2,2}=\R^{1,1}\times\R^{1,1},$ such that first factor corresponds to the coordinates $(x_0,x_1)$ and the second factor to the coordinates $(x_2,x_3);$ the metrics on the factors are thus $-dx_0^2+dx_1^2$ and $-dx_2^2+dx_3^2$ respectively. We also consider the corresponding natural inclusion $SO(1,1)\times SO(1,1)\subset SO(2,2)$. We are first interested in the description of the set 
\[S_{\A}^1:=\Phi^{-1}(SO(1,1)\times SO(1,1))\subset Spin(2,2)\]
where $\Phi$ is the double cover (\ref{cubriente}). To this end, it is convenient to first introduce some $\A$-valued maps, already considered in \cite{konderak}. Let $a\in\A;$ writing 
\[a=\frac{1+\sigma}{2}(u+v)+\frac{1-\sigma}{2}(u-v),\] 
$u,v\in\R,$ and using the properties
\begin{equation}\label{prop_coef}
\left(\frac{1+\sigma}{2}\right)^2=\frac{1+\sigma}{2},\hspace{0.1in}\left(\frac{1-\sigma}{2}\right)^2=\frac{1-\sigma}{2},\hspace{0.1in}\left(\frac{1+\sigma}{2}\right)\left(\frac{1-\sigma}{2}\right)=0,
\end{equation}
we have
 \[a^n=\frac{1+\sigma}{2}(u+v)^n+\frac{1-\sigma}{2}(u-v)^n\]
 for all $n\in\N.$ Thus we can define the exponential map $\A\to\A$ by 
 \begin{equation}\label{exp_num_lorentz}
e^a:=\sum_{n=0}^{\infty}\frac{a^n}{n!}=\frac{1+\sigma}{2}e^{u+v}+\frac{1-\sigma}{2}e^{u-v}
\end{equation} 
for all $a=u+\sigma v\in\A,$ where in the right-hand side $e^{(\cdot)}$ is the usual exponential map, and also define the $\A$-valued hyperbolic sin and cosin functions by the usual formulas 
 \[\cosh(a):=\frac{e^a+e^{-a}}{2}\hspace{.5cm}\mbox{and}\hspace{.5cm}\sinh(a):=\frac{e^a-e^{-a}}{2}. \] 
It is easy to check the following identities 
\begin{equation}\label{cos_sen_num_lorn}\begin{split}
\cosh(a) & =\cosh(u)\cosh(v)+\sigma\sinh(u)\sinh(v),\\
\sinh(a) & =\sinh(u)\cosh(v)+\sigma\cosh(u)\sinh(v)\end{split}
\end{equation}
for all $a=u+\sigma v\in\A.$ Using the definition (\ref{cubriente}) of $\Phi,$ it is easy to get 
\begin{equation}
S_{\A}^1=\{\pm(\cosh(a)+i\sinh(a)I):\ a\in\A \} \subset Spin(2,2);
\end{equation} 
more precisely, writing $a=u+\sigma v \in \A$ and using the identities (\ref{cos_sen_num_lorn}), we get  
\begin{equation*}
\cosh(a)+i\sinh(a)I=(\cosh(v)+\sigma i\sinh(v)I).(\cosh(u)+i\sinh(u)I),
\end{equation*}
and $\Phi(\pm(\cosh(a)+i\sinh(a)I))$ appears to be the transformation of $\R^{2,2}$ which consists of a Lorentz rotation of angle $-2v$ in the first factor $\R^{1,1}$ and of angle $-2u$ in the second factor $\R^{1,1}.$ Thus, setting
\begin{equation}\label{spinp_11}
Spin'(1,1) :=\{\pm(\cosh(v)+\sigma i\sinh(v)I):\ v\in \R\},
\end{equation}
and
\begin{equation}\label{spin_11}
Spin''(1,1)  :=\{\pm(\cosh(u)+i\sinh(u)I):\ u\in \R\}, 
\end{equation}
we have 
\begin{equation}
S_{\A}^1=Spin'(1,1).Spin''(1,1)\simeq Spin'(1,1)\times_{\mathbb{Z}_2}Spin''(1,1)
\end{equation} 
and the double cover 
\begin{equation}
\Phi:S_{\A}^1 \longrightarrow SO(1,1)\times SO(1,1).
\end{equation}
Now, if we consider the spinorial representation $\rho$ of $Spin(2,2)$ restricted to $S_{\A}^1\subset Spin(2,2),$ $\mathbb{H}_0=\Sigma^+\oplus \Sigma^-$ splits into the sum of four complex lines 
\begin{equation}\label{desc_1_rep}
\Sigma^{+}=\Sigma^{++}\oplus\Sigma^{--},\hspace{0.2in}\Sigma^{-}=\Sigma^{+-}\oplus\Sigma^{-+},
\end{equation} where 
$$\Sigma^{++}=(1+\sigma)(\e+iI)(\R\oplus\R J),\hspace{.5cm} \Sigma^{--}=(1+\sigma)(\e-iI)(\R\oplus\R J),$$
$$\Sigma^{+-}=(1-\sigma)(\e-iI)(\R\oplus\R J)\hspace{.5cm}\mbox{and}\hspace{.5cm} \Sigma^{-+}=(1-\sigma)(\e+iI)(\R\oplus\R J)$$
(recall that the complex structure such that the representation is $\C-$linear is given by the right-multiplication by $J$). Note that $e_0\cdot e_1$ acts as $+id$ on $\Sigma^{++}$ and on $\Sigma^{+-},$ and as $-id$ on $\Sigma^{--}$ and on $\Sigma^{-+},$ whereas $e_2\cdot e_3$ acts as $+id$ on $\Sigma^{++}$ and on $\Sigma^{-+},$ and as $-id$ on $\Sigma^{--}$ and on $\Sigma^{+-}.$

Moreover, it is not difficult to show that the representations of $S_{\A}^1$ on $\Sigma^{++},\Sigma^{--},\Sigma^{+-}$ and $\Sigma^{-+}$ are respectively equivalent to the multiplication by $\pm e^{v+u},$ $\pm e^{-v-u},$ $\pm e^{v-u}$ and $\pm e^{-v+u}$ on $\C.$  

\begin{remar}\label{equiv_repres}
Let $\rho_1=\rho_1^+\oplus\rho_1^-$ and $\rho_2=\rho_2^+\oplus\rho_2^-$ be the spinorial representations of $Spin'(1,1)$ and $Spin''(1,1)$ respectively. The representation 
\begin{equation}\label{rep_prod_tensor}
\rho_1\otimes\rho_2\hspace{.5cm}=\hspace{.5cm}\rho_1^+\otimes\rho_2^+\hspace{.3cm}\oplus\hspace{.3cm}\rho_1^-\otimes\rho_2^- \hspace{.3cm}\oplus\hspace{.3cm}\rho_1^+\otimes\rho_2^-\hspace{.3cm}\oplus\hspace{.3cm}\rho_1^-\otimes\rho_2^+
\end{equation}
of $Spin'(1,1)\times Spin''(1,1)$ is also the sum of the natural representations $\pm e^{v+u},$ $\pm e^{-v-u},$ $\pm e^{v-u},$ $\pm e^{-v+u}$ on $\C,$ where $v\in\R$ describes the $Spin'(1,1)$-factor and $u\in\R$ the $Spin''(1,1)$-factor of $Spin'(1,1)\times Spin''(1,1)$ as in (\ref{spinp_11})-(\ref{spin_11}). Thus, the representation 
\begin{eqnarray}\label{repres_grupo_estructura}
Spin'(1,1)\times Spin''(1,1) & \longrightarrow & End_{\C}(\mathbb{H}_0)\\
(g_1,g_2) & \longmapsto & \rho(g):\xi \longmapsto g\xi,\notag
\end{eqnarray}
where $g=g_1g_2\in S_{\A}^1=Spin'(1,1).Spin''(1,1),$ is equivalent to the representation $\rho_1\otimes\rho_2,$ and the decomposition (\ref{desc_1_rep}) of $\Sigma^+$ and $\Sigma^-$ corresponds to (\ref{rep_prod_tensor}).
\end{remar}

\subsection{Spin geometry of a Lorentzian surface in $\R^{2,2}$} 
\subsubsection{Fundamental equations} Let $M$ be a Lorentzian surface in $\R^{2,2}.$ Let us denote  by $E$ its normal bundle and by $B:TM\times TM\to E$ its second fundamental form defined by \[B(X,Y)=\overline{\nabla}_XY-\nabla_XY\] for all $X,Y\in TM,$ where $\nabla$ and $\overline{\nabla}$ are the Levi-Civita connections of $M$ and $\R^{2,2}$ respectively. We moreover assume that $TM$ and $E$ are oriented, both in space and in time: we assume that the bundles $TM$ and $E$ are oriented, and that, for all $p\in M,$ a component of $\{X\in T_pM:\ g(X,X)<0\}$ and a component of $\{X\in E_p:\ g(X,X)<0\}$ are distinguished, in a continuous manner; a vector tangent or normal to $M$ belonging to one of these distinguished components will be called \emph{future-directed}. We will moreover adopt the following convention: a basis $(u,v)$ of $T_pM$ or $E_p$ will be said positively oriented (in space and in time) if it has the orientation of $T_pM$ or $E_p,$  and if $g(u,u)<0$ and $g(v,v)>0$ with $u$ future-directed. Let us denote by $K$ and $K_N$ the curvatures of $M$ and $E$ ($E$ is equipped with the normal connection), and by $\tilde{\nabla}$ the natural connection induced on $T^*M^{\otimes 2}\otimes E.$ If $(e_2,e_3)$ and $(e_0,e_1)$ are orthonormal, positively oriented bases of $TM$ and $E$ respectively, the second fundamental form satisfies the following equations (see e.g. \cite{oneill}):
\begin{enumerate}
\item $K= |B(e_2,e_3)|^2-\la B(e_2,e_2),B(e_3,e_3) \ra$ (Gauss equation),
\item $K_N=\la (S_{e_0}\circ S_{e_1}-S_{e_1}\circ S_{e_0})(e_2),e_3\ra$ (Ricci equation),
\item $(\tilde{\nabla}_XB)(Y,Z)-(\tilde{\nabla}_YB)(X,Z)=0$ (Codazzi equation).
\end{enumerate}
As usual, if $\nu\in E,$ $S_{\nu}$ stands for the symmetric operator on $TM$ such that \[\la S_{\nu}(X),Y\ra=\la B(X,Y),\nu\ra\] for all $X,Y\in TM.$

\begin{remar}\label{teo_fund_inmersion}
Assume that $(M,g)$ is a surface equipped with a Lorentzian metric, and $E$ is a bundle on $M,$ of rank 2, with a fibre Lorentzian metric and a compatible connection. Suppose moreover that $M$ and $E$ are oriented, in space and in time. Then, if $B:TM\times TM\to E$ is a bilinear and symmetric map satisfying the equations $(1),(2)$ and $(3)$ above, the fundamental theorem says that, locally, there is an isometric immersion of $M$ into $\R^{2,2}$ with normal bundle $E$ and second fundamental form $B.$ The immersion is moreover unique up to the rigid motions of $\R^{2,2}.$ We will obtain a spinorial proof of this theorem below (Corollary \ref{two steps integration}).
\end{remar}

\subsubsection{Spinorial Gauss formula}\label{gauss_formula}
We assume here that the tangent and the normal bundles of $M\subset\R^{2,2}$ are oriented (in space and in time), with given spin structures. There is a natural identification between the spinor bundle of $\R^{2,2}$ restricted to $M,$ $\Sigma\R^{2,2}_{|M},$ and the spinor bundle of $M$ twisted by the spinor bundle of $E,$ $\Sigma:=\Sigma E\otimes\Sigma M;$ see \cite{bar} and also Remark \ref{equiv_repres}. Moreover, exactly as in the Riemannian case, we have a spinorial Gauss formula (see \cite{bar,hijazi,Va}): if $\overline{\nabla}$ is the spinorial connection of $\Sigma\R^{2,2}$ and $\nabla$ is the spinorial connection of $\Sigma$ defined by 
\begin{equation*}
\nabla:=\nabla^{\Sigma E}\otimes id_{\Sigma M}+id_{\Sigma E}\otimes\nabla^{\Sigma M}
\end{equation*}
where $\nabla^{\Sigma E}$ and $\nabla^{\Sigma M}$ denote the spinorial connections on $\Sigma E$ and $\Sigma M,$ then, for any $\varphi\in\Sigma$ and any $X\in TM,$ 
\begin{equation}\label{formula_gauss_espinorial}
\overline{\nabla}_X\varphi=\nabla_X\varphi+\frac{1}{2}\sum_{j=2}^3\epsilon_je_j\cdot B(X,e_j)\cdot\varphi
\end{equation}
where $\epsilon_j=\la e_j,e_j\ra,$ and the dot $"\cdot"$ is the Clifford action of $\R^{2,2}.$ Thus, if $\varphi\in\Sigma\R^{2,2}$ is parallel, i.e. is such that $\overline{\nabla}\varphi=0,$  then its restriction to $M$ satisfies
\begin{equation}\label{for_gauss_paralelo}
\nabla_X\varphi=-\frac{1}{2}\sum_{j=2}^3\epsilon_je_j\cdot B(X,e_j)\cdot\varphi.
\end{equation}
Taking the trace, we get the following Dirac equation
\begin{equation}\label{ecua_dirac_inmersion}
D\varphi=\vec{H}\cdot \varphi,
\end{equation}
where $D\varphi:=-e_2\cdot\nabla_{e_2}\varphi+e_3\cdot\nabla_{e_3}\varphi$ and where $\vec{H}=\frac{1}{2}tr_gB\ \in E$ is the mean curvature vector of $M$ in $\R^{2,2}.$ 

\subsection{The inverse construction}
Let $(M,g)$ be a Lorentzian surface and $E$ a bundle of rank 2 on $M,$ equipped with a fibre Lorentzian metric and a compatible connection; we assume that $M$ and $E$ are oriented (in space and in time), with given spin structures. If $(e_0,e_1)$ and $(e_2,e_3)$ are positively oriented and orthonormal frames of $E$ and $TM,$ then, in the respective Clifford bundles, $(e_0\cdot e_1)^2=1$ and $(e_2\cdot e_3)^2=1;$ the spinor  bundles $\Sigma E$ and  $\Sigma M$ thus split into
$$\Sigma E=\Sigma^+E\oplus\Sigma^-E\hspace{.5cm}\mbox{and}\hspace{.5cm} \Sigma M=\Sigma^+M\oplus\Sigma^-M$$
where $e_0\cdot e_1$ acts as $+id$ on $\Sigma^+E$ and as $-id$ on $\Sigma^-E,$ whereas $e_2\cdot e_3$ acts as $+id$ on $\Sigma^+M$ and as $-id$ on $\Sigma^-M.$ We consider the spinor bundle $\Sigma M$ twisted by the spinor bundle $\Sigma E$ and defined by 
\begin{equation*}
\Sigma:=\Sigma E\otimes\Sigma M. 
\end{equation*}
We endow $\Sigma$ with the spinorial connection
\begin{equation*}\nabla:=\nabla^{\Sigma E}\otimes id_{\Sigma M}+id_{\Sigma E}\otimes\nabla^{\Sigma M}.\end{equation*}
We also define the Clifford product $"\cdot"$ by \begin{equation*}
\begin{cases}
X\cdot\varphi=(X\cdot_E\alpha)\otimes\beta \hspace{0.1in}\hspace{.3cm}\text{if}\hspace{0.1in}X\in\Gamma(E),\\
X\cdot\varphi=\overline{\alpha}\otimes(X\cdot_M\beta)\hspace{.3cm}\text{if}\hspace{0.1in}X\in\Gamma(TM),
\end{cases}
\end{equation*}
where $\varphi=\alpha\otimes\beta$ belongs to $\Sigma,$ $\cdot_E$ and $\cdot_M$ denote the Clifford actions on $\Sigma E$ and $\Sigma M$ respectively and where $\overline{\alpha}=\alpha^+-\alpha^-\ \in\ \Sigma E=\Sigma^+E\oplus\Sigma^-E$. Finally we define the Dirac operator
\begin{equation}\label{def dirac}
D\varphi:=-e_2\cdot\nabla_{e_2}\varphi+e_3\cdot\nabla_{e_3}\varphi
\end{equation} 
where $(e_2,e_3)$ is an orthogonal basis tangent to $M$ such that $|e_2|^2=-1$ and $|e_3|^2=1.$  
\\
\\If we denote by $Q_E$ and $Q_M$ the $SO(1,1)$ principal bundles of the oriented and orthonormal frames of $E$ and $TM,$ by $\tilde{Q}_E\to Q_E$ and $\tilde{Q}_M\to Q_M$ the given spin structures on $E$ and $TM,$ and by $p_E:\tilde{Q}_E\to M$ and $p_M:\tilde{Q}_M\to M$ the natural projections, we define the principal bundle over $M$ 
\begin{equation*}\tilde{Q}:=\tilde{Q}_E\times_M\tilde{Q}_M=\{(\tilde{s_1},\tilde{s_2})\in\tilde{Q}_E\times\tilde{Q}_M:p_E(\tilde{s_1})=p_M(\tilde{s_2}) \}.\end{equation*}

\begin{remar}
$\Sigma$ is the vector bundle associated to the principal bundle $\tilde{Q}$ and to the spinor representation $\rho_1\otimes\rho_2\simeq\rho$ of the structure group $Spin'(1,1)\times Spin''(1,1);$ see Remark \ref{equiv_repres}.
\end{remar}
Since the group $S^1_{\A}=Spin'(1,1).Spin''(1,1)$ belongs to $Spin(2,2),$ which preserves the $\A$-bilinear map $H$ defined on $\mathbb{H}_0$ by (\ref{aplic_H_restricta}), the spinor bundle $\Sigma$ is also equipped with a $\A$-bilinear map $H$ and with a real scalar product $\la\cdot,\cdot\ra:=\Re e\ H(\cdot,\cdot)$ of signature $(4,4)$ (here $\Re e$ means that we consider the coefficient of $1$ in the decomposition $\A\simeq\R 1\oplus\R\sigma$). We may also define a $\mathbb{H}_1$-valued scalar product on $\Sigma$ by
\begin{equation}\label{prod_escal_vector}
\laa\psi,\psi'\raa:=\sigma i\ \overline{\xi'}\xi,
\end{equation}
where $\xi$ and $\xi'\in\mathbb{H}_0$ are respectively the components of $\psi$ and $\psi'$ in some local section of  $\tilde{Q}.$ This scalar product is $\A$-bilinear, and satisfies the following properties: for all $\psi,\psi'\in\Sigma$ and for all $X\in E\oplus TM$ 
\begin{equation}\label{prop_prod_escal_vector} 
\laa\psi,\psi'\raa=\overline{\laa\psi',\psi\raa}\hspace{0.2in}\text{and}\hspace{0.2in}\laa X\cdot\psi,\psi'\raa=-\widehat{\laa\psi,X\cdot\psi'\raa}.
\end{equation} 
Note that, by definition, $H(\psi,\psi')$ is the coefficient of $\sigma i\e$ in the decomposition of $\laa\psi,\psi'\raa$ in the basis $\sigma i\e,I,iJ,K$ of $\mathbb{H}_1$ (basis as a module over $\A$), and that (\ref{prop_prod_escal_vector}) yields
\begin{equation}\label{aplic_H_propiedad}
H(\psi,\psi')=H(\psi',\psi) \hspace{0.2in}\text{and}\hspace{0.2in} H(X\cdot\psi,\psi')=\widehat{H(\psi,X\cdot\psi')},
\end{equation}
for all $\psi,\psi'\in\Sigma$ and for all $X\in E\oplus TM.$ In particular, the real scalar product satisfies
 \begin{equation}\label{properties scalar product}
\langle\psi,\psi'\rangle=\langle\psi',\psi\rangle \hspace{0.2in}\text{and}\hspace{0.2in} \langle X\cdot\psi,\psi'\rangle=\langle\psi,X\cdot\psi'\rangle
\end{equation}
for all $\psi,\psi'\in\Sigma$ and for all $X\in E\oplus TM.$
\\

Finally, setting
\begin{equation*}
\Sigma^{++}:=\Sigma^+E\otimes\Sigma^+M,\hspace{0.2in}\Sigma^{--}:=\Sigma^-E\otimes\Sigma^-M,\hspace{0.1in}\Sigma^{+-}:=\Sigma^+E\otimes\Sigma^-M,\hspace{0.2in} \Sigma^{-+}:=\Sigma^-E\otimes\Sigma^+M
\end{equation*}  
and
\begin{equation*}
\Sigma^+:=\Sigma^{++}\oplus\Sigma^{--},\hspace{1cm}\Sigma^-:=\Sigma^{+-}\oplus\Sigma^{-+},
\end{equation*}  
the spinor bundle $\Sigma$ splits into
$$\Sigma=\Sigma^+\oplus\Sigma^-=\Sigma^{++}\oplus\Sigma^{--}\oplus\Sigma^{+-}\oplus\Sigma^{-+}.$$
These splittings correspond to the splittings (\ref{splitting H0}) and (\ref{desc_1_rep})-(\ref{rep_prod_tensor}) of the spinorial representation. 

\subsection{Notation}\label{section notation}
We will use the following notation: if $\tilde{s}\in\tilde{Q}$ is a given spinorial frame, the brackets $[\cdot]$ will denote the coordinates in $\mathbb{H}_0$ of the spinor fields in the frame $\tilde{s},$ that is, for $\varphi\in\Sigma,$
\begin{equation*}
\varphi\simeq[\tilde{s},[\varphi]]\hspace{0.2in}\in\hspace{0.2in}\Sigma\ \simeq\ \tilde{Q}\times\mathbb{H}_0/\rho_1\otimes\rho_2.
\end{equation*}
We will also use the brackets to denote the coordinates in $\tilde{s}$ of the elements of the Clifford algebra $Cl(E\oplus TM):$ $X\in Cl_0(E\oplus TM)$ and $Y\in Cl_1(E\oplus TM)$ will be respectively represented by $[X]\in\mathbb{H}_0$ and $[Y]\in\mathbb{H}_1$ such that, in $\tilde{s},$ 
\begin{equation*}X\simeq\begin{pmatrix}[X] & 0\\ 0 & \widehat{[X]}\end{pmatrix} \hspace{0.2in}\text{and}\hspace{0.2in} Y\simeq\begin{pmatrix}0 & [Y]\\ \widehat{[Y]} & 0\end{pmatrix}.\end{equation*}
Note that
\begin{equation*}
[X\cdot\varphi]=[X][\varphi]\hspace{0.2in}\text{and}\hspace{0.2in}[Y\cdot\varphi]=\sigma i\ [Y]\widehat{[\varphi]},
\end{equation*}
and that, in a spinorial frame $\tilde{s}\in\tilde{Q}$ such that $\pi(\tilde{s})=(e_0,e_1,e_2,e_3),$ where $\pi:\tilde{Q}\to Q_E\times_M Q_M$ is the natural projection onto the bundle of the orthonormal frames of $E\oplus TM$ adapted to the splitting, $e_0,e_1,e_2$ and $e_3\in Cl_1(E\oplus TM)$ are respectively represented by $\sigma i\e, I, iJ$  and $K\in \mathbb{H}_1$ (recall (\ref{aplicliff}) and (\ref{rep_alg})).

\section{Spinor representation of Lorentzian surfaces}\label{section spin representation}
\subsection{The main result} In this section we present the principal theorem concerning the spinor representation of Lorentzian surfaces immersed in $\R^{2,2}.$ This extends to the signature $(2,2)$ the main results of \cite{bayard_lawn_roth} and \cite{bayard1}.

\begin{thm}\label{thm_prin}
Let $(M,g)$ be a simply connected Lorentzian surface and $E$ a Lorentzian bundle of rank 2 on $M$ equipped with a compatible connection. We assume that $M$ and $E$ are oriented (in space and in time), with given spin structures. Let $\Sigma=\Sigma E\otimes\Sigma M$ be the twisted spinor bundle and $D$ its Dirac operator, defined in (\ref{def dirac}). Let $\vec{H}\in \Gamma(E)$ be a section of $E.$ The three following statements are equivalent:
\begin{enumerate}
\item There is a spinor field $\varphi\in\Gamma(\Sigma)$ solution of the Dirac equation
\begin{equation}\label{eqn dirac th}
D\varphi=\vec{H}\cdot\varphi,\hspace{1cm}\mbox{with}\hspace{1cm}H(\varphi,\varphi)=1.
\end{equation}
\item There is a spinor field $\varphi\in\Gamma(\Sigma)$ with $H(\varphi,\varphi)=1$ solution of 
\begin{equation}\label{first killing equation}
\nabla_X\varphi=-\frac{1}{2}\sum_{j=2}^3\epsilon_je_j\cdot B(X,e_j)\cdot\varphi
\end{equation}
for all $X\in TM,$ where $\epsilon_j=g(e_j,e_j)$ and $B:TM\times TM\to E$ is bilinear symmetric with $\frac{1}{2}tr_gB=\vec{H}.$
\item There is an isometric immersion $F$ of $(M,g)$ into $\R^{2,2}$ with normal bundle $E,$ second fundamental form $B$ and mean curvature $\vec{H}.$
\end{enumerate} 
Moreover, $F=\int\xi,$ where $\xi$ is the closed $1$-form on $M$ with values in $\R^{2,2}$ (see Lemma \ref{lem xi R22}) defined by 
\begin{equation*}
\xi(X):=\laa X\cdot\varphi,\varphi\raa
\end{equation*}
for all $X\in TM.$ 
\end{thm}
The claims $(3)\Rightarrow(2)\Rightarrow(1)$ are direct consequences of the spinorial Gauss formula (Section \ref{gauss_formula}). We now prove $(1)\Rightarrow(3)$ using the fundamental theorem of submanifolds (see Remark \ref{teo_fund_inmersion}) and the following
\begin{prop}\label{prop_princ}
Let $M, E,\Sigma$ and $\vec{H}$ as in Theorem \ref{thm_prin}. Assume that there exists a spinor field $\varphi\in\Gamma(\Sigma)$ solution of (\ref{eqn dirac th}). Then the bilinear map $B:TM\times TM\to E$ defined by 
\begin{equation}\label{seg_for_fund}
\la B(X,Y),\nu\ra=-2\la X\cdot\nabla_Y\varphi,\nu\cdot\varphi\ra
\end{equation}
for all $X,Y\in\Gamma(TM)$ and all $\nu\in\Gamma(E)$ is symmetric, satisfies the Gauss, Codazzi and Ricci equations and is such that $\vec{H}=\frac{1}{2}tr_gB.$
\end{prop} 

In the proposition and below, we use the same notation $\la\cdot,\cdot\ra$ to denote the scalar products on $TM,$ on $E,$ and on $\Sigma.$ As in \cite{friedrich} (and after in \cite{lawn}, \cite{lawn_roth}, \cite{morel}, \cite{roth} in codimension one, and in \cite{bayard_lawn_roth} and \cite{bayard1} in codimension two) the proof of this proposition relies on the fact that such a spinor field is necessarily a solution of (\ref{first killing equation}), with this bilinear map $B$:
\begin{lem}\label{lema_princ}
If $\varphi$ is a solution of (\ref{eqn dirac th}), then $\varphi$ solves the Killing type equation (\ref{first killing equation}) where $B$ is the bilinear map defined in (\ref{seg_for_fund}).
\end{lem}
\begin{proof}
We consider the $\A$-module structure $\sigma:=e_0\cdot e_1\cdot e_2\cdot e_3,$ defined on the Clifford bundle $Cl(E\oplus TM)$ by the multiplication on the left, and on spinor bundle $\Sigma$ by the Clifford action. The map $H:\Sigma\times\Sigma\to \A$ is $\A$-bilinear with respect to this $\A$-module structure, whereas the Clifford action satisfies \[\sigma\cdot(X\cdot\varphi)=(\sigma\cdot X)\cdot\varphi=-X\cdot(\sigma\cdot\varphi),\] for all $\varphi\in\Sigma$ and $X\in E\oplus TM.$
Now, we consider the following spinors:
\begin{equation*}\{ \varphi,\ \ e_2\cdot e_3\cdot\varphi,\ \ e_3\cdot e_1\cdot\varphi,\ \ e_1\cdot e_2\cdot\varphi\}.\end{equation*}
Using the identities in (\ref{aplic_H_propiedad}), it is easy to show that these spinors are orthogonal with respect to the form $H,$ with norm $1,-1,1,-1$ respectively; in particular,
\begin{align*} \nabla_X\varphi & = H(\nabla_X\varphi,\varphi)\varphi-H(\nabla_X\varphi,e_2\cdot e_3\cdot\varphi)e_2\cdot e_3\cdot\varphi\\&\  +H(\nabla_X\varphi,e_3\cdot e_1\cdot\varphi)e_3\cdot e_1\cdot\varphi-H(\nabla_X\varphi,e_1\cdot e_2\cdot\varphi)e_1\cdot e_2\cdot\varphi.
\end{align*}
for all $X\in TM.$ We claim that
\begin{equation}\label{two claims to prove}
H(\nabla_X\varphi,\varphi)=0\hspace{1cm}\mbox{and}\hspace{1cm}H(\nabla_X\varphi,e_2\cdot e_3\cdot\varphi)=0.
\end{equation}
The first identity is a direct consequence of $H(\varphi,\varphi)=1.$ The second one is a consequence of the Dirac equation (\ref{eqn dirac th}): assuming that $X=e_2$ (the proof is analogous if $X=e_3$), we have   \begin{align*}
H(\nabla_{e_2}\varphi,e_2\cdot e_3\cdot\varphi) &= \widehat{H(e_2\cdot\nabla_{e_2}\varphi,e_3\cdot\varphi)}= H(e_3^2\cdot\nabla_{e_3}\varphi,\varphi)-\widehat{H(\vec{H}\cdot\varphi,e_3\cdot\varphi)}\\
&= -H(\nabla_{e_3}\varphi,\varphi)-H(\varphi,\vec{H}\cdot e_3\cdot\varphi).
\end{align*} 
But $H(\nabla_{e_3}\varphi,\varphi)=0$ and
$$H(\varphi,\vec{H}\cdot e_3\cdot\varphi)=H(e_3\cdot\vec{H}\cdot\varphi,\varphi)=-H(\vec{H}\cdot e_3\cdot\varphi,\varphi)=-H(\varphi,\vec{H}\cdot e_3\cdot\varphi),$$
that is $H(\varphi,\vec{H}\cdot e_3\cdot\varphi)=0,$ and the second identity in (\ref{two claims to prove}) follows. We thus get 
\begin{equation}\label{killing eqn pr}
\nabla_X\varphi=\eta(X)\cdot\varphi
\end{equation}
with 
$$\eta(X):=H(\nabla_X\varphi,e_3\cdot e_1\cdot\varphi)e_3\cdot e_1-H(\nabla_X\varphi,e_1\cdot e_2\cdot\varphi)e_1\cdot e_2.$$ 
Using the relations $\sigma\cdot e_3\cdot e_1=e_2\cdot e_0$ and $\sigma\cdot e_1\cdot e_2=e_0\cdot e_3,$ we get that $\eta(X)$ has the form
\begin{equation}\label{expresi_eta}
\eta(X)=e_2\cdot \nu_2+e_3\cdot\nu_3,
\end{equation}
for some $\nu_2,\nu_3\in E.$ Now, recalling (\ref{properties scalar product}), for each $\nu \in E$ and $j=2,3,$
\begin{equation*}
\la B(e_j,X),\nu\ra=-2\la e_j\cdot\nabla_X\varphi,\nu\cdot\varphi\ra=-2\la\nabla_X\varphi,e_j\cdot\nu\cdot\varphi\ra= -2\la\eta(X)\cdot\varphi,e_j\cdot\nu\cdot\varphi\ra,
\end{equation*}
which, using (\ref{expresi_eta}), yields 
\begin{equation}\label{seg_form_fun_iguald}
\la B(e_j,X),\nu\ra =-2\la e_2\cdot\nu_2\cdot\varphi,e_j\cdot\nu\cdot\varphi\ra -2\la e_3\cdot\nu_3\cdot\varphi,e_j\cdot\nu\cdot\varphi\ra.
\end{equation}
We note that for all $\nu,\nu'\in E$ we have
\begin{equation}\label{pty proof principal lemma}
\la e_2\cdot e_3\cdot\varphi,\nu\cdot\nu'\cdot\varphi\ra=0;
\end{equation}
the proof is analogous to the proof of Lemma 3.1 in \cite{bayard1} and is omitted here. Thus (\ref{seg_form_fun_iguald}) reads 
$$\la B(e_2,X),\nu\ra=-2\la\nu_2\cdot\varphi,\nu\cdot\varphi\ra=2\la\nu_2,\nu\ra$$
and
$$\la B(e_3,X),\nu\ra=2\la\nu_3\cdot\varphi,\nu\cdot\varphi\ra=-2\la\nu_3,\nu\ra.$$
Indeed, these last identities hold since, for $i=2,3,$ 
\begin{eqnarray*}
\langle\nu_i\cdot\varphi,\nu\cdot\varphi\rangle&=&\langle\nu\cdot\nu_i\cdot\varphi,\varphi\rangle= -\langle\nu_i\cdot\nu\cdot\varphi,\varphi\rangle-2\langle\nu_i,\nu\rangle\langle\varphi,\varphi\rangle\\
&=&-\langle\nu\cdot\varphi,\nu_i\cdot\varphi\rangle-2\langle\nu_i,\nu\rangle
\end{eqnarray*}
and thus $\langle\nu_i\cdot\varphi,\nu\cdot\varphi\rangle=-\langle\nu_i,\nu\rangle.$ Hence $\nu_2=\frac{1}{2}B(e_2,X)$ and $\nu_3=-\frac{1}{2}B(e_3,X),$ and (\ref{killing eqn pr})-(\ref{expresi_eta}) imply formula (\ref{first killing equation}).
\end{proof}

\begin{remar}
The proof given here does not use any decomposition of the spinor fields; using the same ideas, it should be possible to simplify the proofs of \cite[Lemma 3.1]{bayard_lawn_roth} and \cite[Lemma 2.1]{bayard1}.
\end{remar}

\begin{proof}[Proof of Proposition \ref{prop_princ}]
The bilinear map $B$ is symmetric in view of the Dirac equation (\ref{eqn dirac th}) together with the properties (\ref{properties scalar product}) and (\ref{pty proof principal lemma}). The equations of Gauss, Codazzi and Ricci appear to be the integrability conditions of (\ref{first killing equation}): the proof is completely analogous to that given in \cite[Theorem 2]{bayard1}, and is therefore omitted.
\end{proof}

In the next section we prove the second part of Theorem \ref{thm_prin}. 
\subsection{Weierstrass representation}\label{weierstrass_repre}
We assume that  $\varphi\in \Gamma(\Sigma)$ is a spinor field solution of (\ref{eqn dirac th}), and we define the $\mathbb{H}_1$-valued $1$-form $\xi$ by 
\begin{equation}\label{forma_xi}
\xi(X)=\laa X\cdot\varphi,\varphi\raa\hspace{0.1in}\in\hspace{0.1in}\mathbb{H}_1
\end{equation} 
for $X\in TM\oplus E,$ where the pairing $\laa.,.\raa:\Sigma\times\Sigma\to \mathbb{H}_1$ is defined in (\ref{prod_escal_vector}).

\begin{lem}\label{lem xi R22}
The $1$-form $\xi$ satisfies the following properties:
\begin{enumerate}
\item $\xi=-\widehat{\overline{\xi}},$ that is $\xi$ takes its values in $\R^{2,2}\subset\mathbb{H}_1;$
\item $\xi:TM\rightarrow \R^{2,2}$ is closed, that is $d\xi=0.$
\end{enumerate} 
\end{lem}
\begin{proof}
\emph{1.} Using the properties (\ref{prop_prod_escal_vector}), we get 
\begin{equation*}
\xi(X)= \laa X\cdot\varphi,\varphi\raa=-\widehat{\laa \varphi,X\cdot\varphi\raa}=-\widehat{\overline{\laa X\cdot\varphi,\varphi \raa}}=-\widehat{\overline{\xi(X)}};
\end{equation*}
the result then follows from (\ref{iden_espa}).
\\\emph{2.} By a straightforward computation, we get 
\begin{equation*} d\xi(e_2,e_3)=\laa e_3\cdot\nabla_{e_2}\varphi,\varphi\raa-\laa e_2\cdot\nabla_{e_3}\varphi,\varphi\raa+\laa e_3\cdot\varphi,\nabla_{e_2}\varphi\raa-\laa e_2\cdot\varphi,\nabla_{e_3}\varphi \raa. 
\end{equation*} 
Now, the last two terms satisfy
\begin{equation*} \laa e_3\cdot\varphi,\nabla_{e_2}\varphi\raa=-\widehat{\overline{\laa e_3\cdot\nabla_{e_2}\varphi,\varphi\raa}}\hspace{.5cm}\mbox{and}\hspace{.5cm} \laa e_2\cdot\varphi,\nabla_{e_3}\varphi\raa=-\widehat{\overline{\laa e_2\cdot\nabla_{e_3}\varphi,\varphi\raa}}.\end{equation*}  
Moreover
\begin{align*}
\laa e_3\cdot\nabla_{e_2}\varphi,\varphi\raa-\laa e_2\cdot\nabla_{e_3}\varphi,\varphi\raa &= -\widehat{\laa e_2\cdot e_3\cdot\nabla_{e_2}\varphi,e_2\cdot\varphi\raa}-\widehat{\laa e_3\cdot e_2\cdot\nabla_{e_3}\varphi,e_3\cdot\varphi\raa}\\
&= \laa e_2\cdot\nabla_{e_2}\varphi-e_3\cdot\nabla_{e_3}\varphi,e_2\cdot e_3\cdot\varphi\raa \\
&= -\laa D\varphi,e_2\cdot e_3\cdot\varphi\raa \\
&= -\laa \vec{H}\cdot\varphi,e_2\cdot e_3\cdot\varphi\raa,
\end{align*} 
and thus 
\begin{equation*} d\xi(e_2,e_3)=-\laa \vec{H}\cdot\varphi,e_2\cdot e_3\cdot\varphi\raa+\widehat{\overline{\laa \vec{H}\cdot\varphi,e_2\cdot e_3\cdot\varphi\raa}}.
\end{equation*}
Noting finally that
\begin{align*}
\laa \vec{H}\cdot\varphi,e_2\cdot e_3\cdot\varphi\raa &= -\widehat{\laa \varphi,\vec{H}\cdot e_2\cdot e_3\cdot\varphi\raa}= -\widehat{\laa \varphi, e_2\cdot e_3\cdot\vec{H}\cdot\varphi\raa}\\
&= \widehat{\laa  e_2\cdot e_3\cdot\varphi, \vec{H}\cdot\varphi\raa}=\widehat{\overline{\laa \vec{H}\cdot\varphi,e_2\cdot e_3\cdot\varphi\raa}},
\end{align*} 
we get that $d\xi=0.$
\end{proof}

We now assume that $M$ is simply connected; then, there exists a function $F:M\to \R^{2,2}$ such that $dF=\xi.$ The next theorem follows from the properties of the Clifford action and its proof is analogous to the proof of \cite[Theorem 3]{bayard1}, and is therefore omitted.   

\begin{thm}\label{teorema_isometria}
1. The map $F:M\to \R^{2,2}$ is an isometry.
\\2. The map 
\begin{align*}
\Phi_E:\hspace{1cm} E & \longrightarrow M\times\R^{2,2} \\
 X\in E_m & \longmapsto (F(m),\xi(X))
\end{align*}
is an isometry between $E$ and the normal bundle $N(F(M))$ of $F(M)$ in $\R^{2,2},$ preserving connections and second fundamental forms.
\end{thm}

\begin{remar}
If $M$ is a Lorentzian surface in $\R^{2,2},$ the immersion may be obtained from the constant spinor fields $\sigma\e$ or $-\sigma\e\in\mathbb{H}_0$ restricted to the surface: indeed, for one of these spinor fields $\varphi,$ and for all $X\in TM\subset M\times\R^{2,2},$ we have 
\begin{equation*}
\xi(X)=\laa X\cdot\varphi,\varphi\raa=-\overline{[\varphi]}[X]\widehat{[\varphi]}=[X],
\end{equation*}
where here the brackets $[X]\in\mathbb{H}_1$ and $[\varphi]=\pm\sigma\e\in\mathbb{H}_0$ represent $X$ and $\varphi$ in one of the two spinorial frames of $\R^{2,2}$ which are above the canonical basis (recall (\ref{prod_escal_vector}) and Section \ref{section notation}). Identifying $[X]\in\R^{2,2}\subset\mathbb{H}_1$ to $X\in\R^{2,2},$ $F=\int\xi$ identifies to the identity.
\end{remar}
Similarly to the Euclidean and Minkowski cases  (\cite{bayard_lawn_roth} and \cite{bayard1}), we deduce a spinorial proof of the fundamental theorem given in Remark \ref{teo_fund_inmersion}: 
\begin{coro}\label{two steps integration}
We may integrate the Gauss, Ricci and Codazzi equations in two steps:
\begin{enumerate}
\item first solving 
\begin{equation} \label{eqn first step}
\nabla_X\varphi=\eta(X)\cdot\varphi
\end{equation}
where
\begin{equation*}
\eta(X)=-\frac{1}{2}\sum_{j=2}^3\epsilon_je_j\cdot B(X,e_j),
\end{equation*}
(there is a solution $\varphi$ in $\Gamma(\Sigma)$ such that $H(\varphi,\varphi)=1,$ unique up to the natural right-action of $Spin(2,2)$ on $\Gamma(\Sigma)$),

\item then solving 
\begin{equation}\label{eqn second step}
dF=\xi 
\end{equation} 
where $\xi(X)=\laa X\cdot\varphi,\varphi\raa$ (the solution is unique, up to translations in $\R^{2,2}\subset\mathbb{H}_1$).
\end{enumerate}
\end{coro}
The Gauss, Ricci and Codazzi equations are in fact exactly the integrability conditions for (\ref{eqn first step}) (see  \cite{bayard1,bayard_lawn_roth} for details); Theorem \ref{teorema_isometria} then shows that the solution $F$ of (\ref{eqn second step}) is an immersion preserving fundamental forms and connections. We finally note that the multiplication on the right by a constant belonging to $Spin(2, 2)$ in the first step, and the addition of a constant belonging to $\R^{2,2}$ in the second step, correspond to a rigid motion in $\R^{2,2}.$ 

\subsection{Lorentzian surfaces in $\R^{1,2}$ and $\R^{2,1}$}
The aim of this section is to deduce spinor characterisations for immersions of Lorentzian surfaces in $\R^{1,2}$ and $\R^{2,1};$ we obtain characterisations which are different to the characterisations given by M.-A. Lawn \cite{lawn_thesis,lawn} and by M.-A. Lawn and J. Roth \cite{lawn_roth}. Keeping the notation of Section \ref{prelim}, we consider the map $\beta:\mathbb{H}_0 \longrightarrow\mathbb{H}_0$ given by $\beta(\xi)=i\sigma\xi I.$ This map is $\A$-linear and satisfies 
$$\beta^2=id_{\mathbb{H}_0}\hspace{.5cm}\mbox{and}\hspace{.5cm}\beta(\xi J)=-\beta(\xi)J$$ 
for all $\xi\in\mathbb{H}_0;$ $\beta$ is thus a real structure on $\mathbb{H}_0.$ We note that $\beta$ is $Spin(2,2)$-equivariant, and thus induces a real structure $\beta:\Sigma\to\Sigma$ on the spinor bundle: it satisfies 
$$\beta^2=id_{\Sigma}\hspace{.5cm}\mbox{and}\hspace{.5cm} \beta(i\varphi)=-i\beta(\varphi)$$
for all $\varphi$ belonging to $\Sigma$ (here $i$ stands for the natural complex structure on $\Sigma;$ in coordinates, this is the right-multiplication by $J,$ see Section \ref{prelim}). Moreover, $\beta$ is anti-linear with respect to the Clifford action of $E\oplus TM$: for all $X\in E\oplus TM$ and $\varphi\in\Sigma,$
\begin{equation*}
\beta(X\cdot\varphi)=-X\cdot\beta(\varphi).
\end{equation*}
Finally, for all $\varphi=\varphi^++\varphi^-\ \in\ \Sigma=\Sigma^+\oplus\Sigma^-$ and all $X\in TM,$ we have
\begin{equation}\label{estruc_real_1}
\beta(\varphi^{\pm})=\beta(\varphi)^{\pm}, \ H(\beta(\varphi),\beta(\varphi))=-H(\varphi,\varphi)\hspace{.3cm}\mbox{and}\hspace{.3cm} \nabla_X\beta(\varphi)=\beta(\nabla_X\varphi).
\end{equation}  

Throughout the section, we suppose that the bundle $E$ is flat, i.e. is of the form $E=\R e_0\oplus \R e_1$ where $e_0$ and $e_1$ are unit, orthogonal and parallel sections of $E$ such that $\la e_0,e_0\ra=-1$ and $\la e_1,e_1\ra=1;$ we moreover assume that $e_0$ is future-directed, and that $(e_0,e_1)$ is positively oriented. We consider the isometric embeddings of $\R^{1,2}$ and $\R^{2,1}$ in $\R^{2,2}\subset \mathbb{H}_1$ given by 
\begin{equation*}
\R^{1,2}=(\sigma i\e)^{\perp}\hspace{0.2in}\text{and}\hspace{0.2in}\R^{2,1}=(I)^{\perp}, 
\end{equation*} 
where $\sigma i\e$ and $I$ are the first two vectors of the canonical basis of $\R^{2,2}\subset \mathbb{H}_1.$ We note that the signatures of $\R^{1,2}$ and $\R^{2,1}$ are $(+,-,+)$ and $(-,-,+)$ respectively. Let $\vec{H}$ be a section of $E$ and $\varphi\in\Gamma(\Sigma)$ be a solution of (\ref{eqn dirac th}). According to Section \ref{weierstrass_repre}, the spinor field $\varphi$ defines an isometric immersion $F:M\rightarrow\R^{2,2}$ (unique, up to translations), with normal bundle $E$ and mean curvature vector $\vec{H}.$ We give a characterisation of the isometric immersions in $\R^{1,2}$ and $\R^{2,1}$ (up to translations) in terms of $\varphi:$ 

\begin{prop}\label{carac_inm_isom}
1- Assume that \begin{equation}\label{inm_iso_R12}
\vec{H}=He_1\hspace{0.3in}\text{and}\hspace{0.3in} e_0\cdot\varphi=\varphi. 
\end{equation}
Then the isometric immersion $F:M\rightarrow\R^{2,2}$ belongs to $\R^{1,2}.$
\\
\\2- Assume that \begin{equation}\label{inm_iso_R21}
\vec{H}=He_0\hspace{0.3in}\text{and}\hspace{0.3in} e_1\cdot\varphi=-\beta(\varphi). 
\end{equation}
Then the isometric immersion $F:M\rightarrow\R^{2,2}$ belongs to $\R^{2,1}.$

Reciprocally, if $F:M\rightarrow\R^{2,2}$ belongs to $\R^{1,2}$ (resp. to $\R^{2,1}$), then  the normal bundle $E$ is flat and (\ref{inm_iso_R12}) (resp. (\ref{inm_iso_R21})) holds for some unit, orthogonal and parallel sections $(e_0,e_1)$ of $E.$
\end{prop}

\begin{proof}
1- Assuming that (\ref{inm_iso_R12}) holds, we compute  $\la\la e_0\cdot\varphi,\varphi\ra\ra=\sigma i\e.$ Thus, the constant vector $\sigma i\e\in\R^{2,2}\subset\mathbb{H}_1$ is normal to the immersion (by Theorem \ref{teorema_isometria}, (2), since this is $\xi(e_0)$), and the immersion thus belongs to $\R^{1,2}.$
\\
\\2- Analogously, assuming that (\ref{inm_iso_R21}) holds, we have \[\la\la e_1\cdot\varphi,\varphi\ra\ra=-\la\la\beta(\varphi),\varphi\ra\ra=-\sigma i\overline{[\varphi]}[\beta(\varphi)] =-\sigma i\overline{[\varphi]}[\varphi]\sigma iI=I\] where $[\varphi]\in\mathbb{H}_0$ represents the spinor field $\varphi$ in some frame $\tilde{s}\in\tilde{Q}.$ The constant vector $I\in\R^{2,2}\subset\mathbb{H}_1$ is thus normal to the immersion, and the result follows.  
\\
\\For the converse statements, we choose $(e_0,e_1)$ such that $\la\la e_0\cdot\varphi,\varphi\ra\ra=\sigma i\e$ in the first case and such that $\la\la e_1\cdot\varphi,\varphi\ra\ra=I$ in the second case. Writing these identities in some frame $\tilde{s},$ we easily deduce (\ref{inm_iso_R12}) and (\ref{inm_iso_R21}).
\end{proof}

Let $\mathcal{H}\subset \R^{2,2}$ be the hyperplane $\R^{1+r,2-r}$ with $r=0,1$ (that is $\mathcal{H}$ is $\R^{1,2}$ if $r=0$ and is $\R^{2,1}$ if $r=1$). If we assume that $M\subset\mathcal{H}\subset \R^{2,2},$ and if we consider $e_0$ and $e_1$ timelike and spacelike unit vector fields such that 
\begin{equation*}
\R^{2,2}=\R e_r\oplus_{\perp} T\mathcal{H}\hspace{0.3in} \text{and}\hspace{0.3in} T\mathcal{H}=\R e_{1-r}\oplus_{\perp} TM,
\end{equation*}
then the intrinsic spinors of $M$ identify with the spinors of $\mathcal{H}$ restricted to $M,$ which in turn identify with the positive spinors of $\R^{2,2}$ restricted to $M:$ this identification is the content of Propositions \ref{iso_haces_1} and \ref{iso_haces_2} below, which, together with the previous results, will give the representation of surfaces in $\R^{1,2}$ and $\R^{2,1}$ by means of spinors of $\Sigma M$ only.

\subsubsection{Lorentzian surfaces in $\R^{1,2}$}\label{superf_R12}
We first deduce from Proposition \ref{carac_inm_isom} \emph{1-} a spinor representation for Lorentzian surfaces in $\R^{1,2}.$ We can define a scalar product on $\C^2$ by setting
\begin{equation*}
\left\la\begin{pmatrix}
a+ib \\ c+id
\end{pmatrix},\begin{pmatrix}
a'+ib' \\ c'+id'
\end{pmatrix}\right\ra:=\frac{ad'+a'd-bc'-b'c}{2};
\end{equation*} 
it is of signature $(2,2).$ This scalar product is $Spin(1,1)$-invariant (the action of $\pm e^u\in Spin(1,1)$ is the multiplication by $\pm e^u$ on the first and by $\pm e^{-u}$ on the second component of the spinors) and thus induces a scalar product $\la.,.\ra$ on the spinor bundle $\Sigma M.$ It satisfies the following properties: for all $\psi,\psi'\in \Sigma M$ and all $X\in TM,$
\begin{equation}
\la\psi,\psi'\ra=\la\psi',\psi\ra\hspace{0.2in}\text{and}\hspace{0.2in}\la X\cdot\psi,\psi'\ra=-\la\psi,X\cdot\psi'\ra.
\end{equation}
This is the scalar product on $\Sigma M$ that we use in this section (and in this section only). We moreover define $|\psi|^2:=\la\psi,\psi\ra.$ The following proposition is analogous to \cite[Proposition 6.1]{bayard_lawn_roth} (see also \cite[Proposition 2.1]{morel}, and the references there), and is proved in \cite{VP}:
\begin{prop}\label{iso_haces_1}
There is an identification \begin{align*}
\Sigma M & \overset{\sim}{\longmapsto}  \Sigma^+_{|M}\\
\psi & \longmapsto  \psi^*,
\end{align*} 
$\C-$linear, and such that, for all $X\in TM$ and all $\psi\in\Sigma M,$ $(\nabla_X\psi)^*=\nabla_X\psi^*,$ the Clifford actions are linked by $(X\cdot\psi)^*=X\cdot e_1\cdot\psi^*$ and 
\begin{equation}\label{norma_R12}
H(\psi^*,\psi^*)=\frac{1+\sigma}{2}|\psi|^2. 
\end{equation}
\end{prop}
Using this identification, the intrinsic Dirac operator on $M,$ defined by 
\begin{equation*}
D\psi:=-e_2\cdot\nabla_{e_2}\psi+e_3\cdot\nabla_{e_3}\psi,
\end{equation*}
is linked to the Dirac operator on $\Sigma$ by 
\begin{equation*}
(D\psi)^*=-e_1\cdot D\psi^*.
\end{equation*} 
We suppose that $\varphi$ is a solution of (\ref{eqn dirac th}) such that (\ref{inm_iso_R12}) holds (the immersion belongs to $\R^{1,2}$), and we choose $\psi\in \Sigma M$ such that $\psi^*=\varphi^+;$ in view of (\ref{norma_R12}), it is such that $|\psi|^2=1$ since $H(\varphi^+,\varphi^+)=\frac{1+\sigma}{2}$ if $H(\varphi,\varphi)=1$ (this last claim relies on (\ref{inm_iso_R12}) together with (\ref{exp_expli_pos})-(\ref{exp_expli_neg})); moreover, it satisfies 
\begin{equation*}
(D\psi)^*=-e_1\cdot D\psi^*= -e_1\cdot \vec{H}\cdot\psi^*=-e_1\cdot He_1\cdot\psi^*=H\psi^*;
\end{equation*} 
$\psi$ is thus a solution of
\begin{equation}\label{ecua_diracM_1}
D\psi=H\psi\hspace{0.2in}\text{with}\hspace{0.2in}|\psi|^2=1. 
\end{equation} 
Reciprocally, if $\psi\in\Sigma M$ is a solution of (\ref{ecua_diracM_1}), we can define $\varphi^+:=\psi^*$ and $\varphi^-:=e_0\cdot\psi^*,$ and get $\varphi:=\varphi^++\varphi^-\in\Sigma,$ a solution of (\ref{eqn dirac th}) with $\vec{H}=He_1$ (we recall that $e_0$ and $e_1$ are parallel sections of $E$ such that $\langle e_0,e_0\rangle=-1$ and $\langle e_1,e_1\rangle=1$); since $e_0\cdot\varphi=\varphi$ we obtain an isometric immersion of $M$ in $\R^{1,2}$ (Proposition \ref{carac_inm_isom} \emph{1-}). A solution of (\ref{ecua_diracM_1}) is thus equivalent to an isometric immersion in $\R^{1,2}.$ We thus obtain a spinorial characterisation of an isometric immersion of a Lorentzian surface in $\R^{1,2},$ which is simpler than the characterisation obtained in \cite{lawn}, where two spinor fields are involved.

\begin{remar}
We also obtain an explicit representation formula: for all $\psi\in\Sigma M,$ we denote by $\alpha(\psi)$ the spinor field whose coordinates in a given spinorial frame are the complex conjugates of the coordinates of $\psi$ in this frame, and by $\overline{\psi}:=\psi^+-\psi^-,$ the usual conjugation in $\Sigma M.$ If we suppose that $\psi\in\Sigma M$ is a solution of (\ref{ecua_diracM_1}), setting $\chi:=\overline{\psi}$ we can show that 
\begin{equation}\label{base_sigmaM1}
\chi,\ \alpha(\chi),\ i\chi,\ i\alpha(\chi) 
\end{equation} 
is $\la.,.\ra$-orthonormal with signature $(-,+,-,+),$ and in particular is a real basis of $\Sigma M$ ($i$ is the natural complex structure of $\Sigma M,$ which is such that the Clifford action is $\C-$linear). For all $X\in TM$ and $\varphi=\psi^*+e_0\cdot\psi^*,$ where $\psi\in\Sigma M$ satisfies (\ref{ecua_diracM_1}), a computation yields \begin{align*}
\xi(X)& =\la\la X\cdot\varphi,\varphi\ra\ra\\ 
&=-\la X\cdot\psi,\alpha(\chi)\ra\ I+\la X\cdot\psi,i\chi\ra\ (iJ)-\la X\cdot\psi,i\alpha(\chi)\ra\ K.
\end{align*}We note that $\la X\cdot\psi,\chi\ra=0,$ and thus that $\xi(X)$ may be interpreted as the coordinates of $X\cdot\psi$ in the orthonormal basis (\ref{base_sigmaM1}). The formula $F=\int \xi$ represents the immersion. For sake of brevity we don't include the proof, and refer to \cite{VP} for details.
\end{remar} 

\subsubsection{Lorentzian surfaces in $\R^{2,1}$}
In this section we deduce from Proposition \ref{carac_inm_isom} \emph{2-} a spinor representation for Lorentzian surfaces in $\R^{2,1}.$ We consider here the following scalar product on $\Sigma M,$ given in coordinates by
\begin{equation*}
\left\la \begin{pmatrix}
a+ib \\ c+id
\end{pmatrix},\begin{pmatrix}
a'+ib' \\ c'+id'
\end{pmatrix}\right\ra:=-\frac{ac'+a'c+bd'+b'd}{2};
\end{equation*} 
it is of signature $(2,2).$ Moreover, for all $\psi,\psi'\in \Sigma M$ and all $X\in TM$ we have
\begin{equation}\la\psi,\psi'\ra=\la\psi',\psi\ra\hspace{0.2in}\text{and}\hspace{0.2in}
\la X\cdot\psi,\psi'\ra=\la\psi,X\cdot\psi'\ra. 
\end{equation}
We moreover write $|\psi|^2:=\la\psi,\psi\ra$ and still denote by $i$ the natural complex structures on $\Sigma$ and on $\Sigma M.$
\begin{prop}\label{iso_haces_2}
There is an identification \begin{align*}
\Sigma M & \overset{\sim}{\longmapsto}  \Sigma^+_{|M}\\
\psi & \longmapsto  \psi^*,
\end{align*} 
$\C-$linear, and such that, for all $X\in TM$ and all $\psi\in\Sigma M,$ $(\nabla_X\psi)^*=\nabla_X\psi^*,$ the Clifford actions are linked by $(X\cdot\psi)^*=ie_0\cdot X\cdot\psi^*$ and 
\begin{equation}\label{norma_R21}
H(\psi^*,\psi^*)=-\frac{1+\sigma}{2}|\psi|^2. 
\end{equation}
\end{prop}
\noindent
The detailed proof is given in \cite{VP}. Using this identification, we have
\begin{equation*}
(D\psi)^*=ie_0\cdot D\psi^*
\end{equation*} 
 for all $\psi\in\Sigma M.$ If we suppose that $\varphi$ is a solution of (\ref{eqn dirac th}), we can choose $\psi\neq 0\in \Sigma M$ such that $\psi^*=\varphi^+;$ moreover, if (\ref{inm_iso_R21}) holds, $\psi$ satisfies 
\begin{equation*}
(D\psi)^*= ie_0\cdot D\psi^*=ie_0\cdot \vec{H}\cdot\psi^*=ie_0\cdot He_0\cdot\psi^*=iH\psi^*,
\end{equation*}
and, using (\ref{estruc_real_1}) and (\ref{norma_R21}), 
\begin{equation}\label{ecua_diracM_2}
D\psi=iH\psi,\hspace{0.3in}|\psi|^2=-1.
\end{equation} 
Reciprocally, if we suppose that $\psi\in\Sigma M$ satisfies (\ref{ecua_diracM_2}), we can define $\varphi^+:=\psi^*$ and $\varphi^-:=e_1\cdot\beta(\psi^*),$ and set $\varphi:=\varphi^++\varphi^-\in\Sigma;$ using (\ref{estruc_real_1}), it is not difficult to see that $\varphi$ satisfies (\ref{eqn dirac th}), and since $e_1\cdot\varphi=-\beta(\varphi),$ defines an isometric immersion of $M$ into $\R^{2,1}$ (Proposition \ref{carac_inm_isom} \emph{2-}). A solution of (\ref{ecua_diracM_2}) is thus equivalent to an isometric immersion of the Lorentzian surface into $\R^{2,1}.$ Here again, we obtain a spinor characterisation of an isometric immersion of a Lorentzian surface in $\R^{2,1},$ which is simpler than the characterisation obtained in \cite{lawn_roth} where two spinor fields are needed.

\begin{remar}
We also obtain an explicit representation formula: for $\psi\in\Sigma M,$ we may consider $\overline{\psi}$ and $\alpha(\psi)$ as in the previous section, and show that 
\begin{equation}\label{base_sigmaM2}
\alpha(\psi),\ i\overline{\psi},\ i\alpha(\psi),\ \overline{\psi}
\end{equation}  
is $\la.,.\ra$-orthonormal with signature $(-,+,-,+);$ in particular this is a real basis of $\Sigma M.$ Setting $\varphi:=\psi^*+e_1\cdot\beta(\psi^*)$ where $\psi\in\Sigma M$ is a solution of (\ref{ecua_diracM_2}), a computation yields 
\begin{align*}
\xi(X)& =\la\la X\cdot\varphi,\varphi\ra\ra\\ 
&=\la X\cdot\psi,\alpha(\psi)\ra\ \sigma i\e-\la X\cdot\psi,i\alpha(\psi)\ra\ iJ-\la X\cdot\psi,\overline{\psi}\ra\ K
\end{align*}
for all $X\in TM.$ Since $\la X\cdot\psi,i\overline{\psi}\ra=0,$ $\xi(X)$ may be interpreted as the coordinates of $X\cdot\psi$ in the orthonormal basis (\ref{base_sigmaM2}). Finally, $F=\int\xi$ represents the immersion.
\end{remar} 

\section{Spinor description of flat Lorentzian surfaces in $\R^{2,2}$}\label{section flat surfaces}
\subsection{The Grassmannian of the Lorentzian planes in $\R^{2,2}$}\label{section grassmannian}
The Grassmannian of the oriented Lorentzian planes in $\R^{2,2}$ identifies to
$$\Q=\{u_1\cdot u_2:\ u_1,u_2\in\R^{2,2},\ |u_1|^2=-|u_2|^2=-1\}\ \subset\ Cl_0(2,2).$$
Setting
\begin{equation*}
\Im m\ \mathbb{H}_0\ :=\ i\A I\ \oplus\ \A J\ \oplus\ i\A K
\end{equation*}
and since $e_2\cdot e_3\simeq iI, e_3\cdot e_1\simeq J$ and $e_1\cdot e_2\simeq iK$ in the identification  $Cl_0(2,2)\simeq \mathbb{H}_0$ given in (\ref{elempar}), we easily get  
\begin{equation*}
\Q\simeq\{p\in\Im m\ \mathbb{H}_0:\ H(p,p)=-1\};
\end{equation*} 
we moreover note that the tangent space of $\Q$ at a point $p$ is explicitly given by
$$T_p{\Q}\simeq\{\xi\in \Im m\ \mathbb{H}_0:\ H(p,\xi)=0\}.$$
We define the \textit{cross product} of two vectors $\xi,\xi'\in \Im m\ \mathbb{H}_0$ by
\begin{equation*} \xi\times\xi':=\frac{1}{2}(\xi\xi'-\xi'\xi)\in\Im m\ \mathbb{H}_0.
\end{equation*} 
It is such that
\begin{equation*}
\laa\xi,\xi' \raa=\sigma i\ H(\xi,\xi')\e+\sigma i\ \xi\times\xi'
\end{equation*}
for all $\xi,\xi'\in \Im m\ \mathbb{H}_0.$ We also define the \textit{mixed product} of three vectors $\xi,\xi',\xi''\in \Im m\ \mathbb{H}_0$ by
\begin{equation*} [\xi,\xi',\xi'']:=H(\xi\times\xi',\xi'')\in \A;
\end{equation*} 
it is also easily seen to be, up to sign, the determinant of the vectors $\xi,\xi',\xi''\in\Im m\ \mathbb{H}_0$ in the basis $(iI,J,iK)$ of $\Im m\ \mathbb{H}_0$ (considered as a $\A$-module). The mixed product is a $\A$-valued volume form on $\Im m\ \mathbb{H}_0,$ and induces a natural \textit{$\A$-valued area form} $\omega_{\Q}$ on $\Q$ by 
\begin{equation}\label{def omega Q}
\omega_{\Q}(p)(\xi,\xi'):=[\xi,\xi',p],
\end{equation} 
for all $p\in\Q$ and all $\xi,\xi'\in T_p\Q.$ 

\subsection{Lorentz surfaces and Lorentz numbers}\label{appendix lorentz}
In this section we present elementary results concerning Lorentz surfaces and Lorentz numbers. We will say that a surface $M$ is a Lorentz surface if there is a covering by open subsets $M=\cup_{\alpha\in S}U_{\alpha}$ and charts
$$\varphi_{\alpha}:\hspace{.3cm}U_{\alpha}\ \rightarrow\ \mathcal{A},\hspace{.5cm} \alpha\in S$$ 
such that the transition functions 
$$\varphi_{\beta}\circ\varphi_{\alpha}^{-1}:\hspace{.3cm}\varphi_{\alpha}(U_\alpha\cap U_\beta)\subset\mathcal{A}\ \rightarrow\ \varphi_{\beta}(U_\alpha\cap U_\beta)\subset\mathcal{A},\hspace{.5cm}\alpha,\ \beta\in S$$
are conformal maps in the following sense: for all $a\in\varphi_{\alpha}(U_\alpha\cap U_\beta)$ and $h\in\mathcal{A},$
$$d\ (\varphi_{\beta}\circ\varphi_{\alpha}^{-1})_a\ (\sigma\ h)\hspace{.3cm}=\hspace{.3cm}\sigma\ d\ (\varphi_{\beta}\circ\varphi_{\alpha}^{-1})_a\ (h).$$
A Lorentz structure is also equivalent to a smooth family of maps 
$$\sigma_x:\hspace{.3cm}T_xM\ \rightarrow\ T_xM,\hspace{.5cm} \mbox{with}\hspace{.5cm} \sigma_x^2=id_{T_xM},\ \sigma_x\neq\pm id_{T_xM}.$$
This definition coincides with the definition of a Lorentz surface given in \cite{Weinstein}: a Lorentz structure is equivalent to a conformal class of Lorentzian metrics on the surface, that is to a smooth family of cones in every tangent space of the surface, with distinguished lines. Indeed, the cone at $x\in M$ is
$$Ker(\sigma_x-id_{T_xM})\ \cup\ Ker(\sigma_x+id_{T_xM})$$
where the sign of the eigenvalues $\pm 1$ permits to distinguish one of the lines from the other. 

If $M$ is moreover oriented, we will say that the Lorentz structure is compatible with the orientation of $M$ if the charts $\varphi_\alpha: U_{\alpha}\rightarrow\mathcal{A},\ \alpha\in S$ preserve the orientations (the positive orientation in $\mathcal{A}=\{u+\sigma v,\ u,v\in\R\}$ is naturally given by $(\partial_u,\partial_v$)). In that case, the transition functions are conformal maps $\mathcal{A}\rightarrow\mathcal{A}$ preserving orientation.

If $M$ is a Lorentz surface, a smooth map $\psi:M\rightarrow \mathcal{A}$ (or $\mathcal{A}^n,$ or a Lorentz surface) will be said to be a conformal map if $d\psi$ preserves the Lorentz structures, that is if
$$d\psi_x(\sigma_xh)\ =\ \sigma_{\psi(x)}(d\psi_x(h))$$
for all $x\in M$ and $h\in T_xM.$ In a chart $\mathcal{A}=\{u+\sigma v,\ u,v\in\R\},$ a conformal map satisfies 
\begin{equation}\label{eqn crl}
\frac{\partial \psi}{\partial v}\ =\ \sigma\ \frac{\partial \psi}{\partial u}.
\end{equation}
Defining the coordinates $(s,t)$ such that
\begin{equation}\label{def s t}
u+\sigma\ v\ =\ \frac{1+\sigma}{2}\ s+\frac{1-\sigma}{2}\ t
\end{equation}
($s$ and $t$ are parameters along the distinguished lines) and writing
$$\psi\ =\ \frac{1+\sigma}{2}\ \psi_1+\frac{1-\sigma}{2}\ \psi_2$$
with $\psi_1,\psi_2\in\R,$ (\ref{eqn crl}) reads $\partial_t\psi_1=\partial_s\psi_2=0,$ and we get 
$$\psi_1=\psi_1(s)\hspace{1cm}\mbox{and}\hspace{1cm} \psi_2=\psi_2(t).$$
A conformal map is thus equivalent to two functions of one variable. We finally note that if $\psi:M\rightarrow\mathcal{A}^n$ is a conformal map, we have, in a chart $a:\mathcal{U}\subset\mathcal{A}\rightarrow M,$
$$d\psi\ =\ \psi' da,$$
where $da=du+\sigma dv$ and $\psi'$ belongs to $\mathcal{A}^n;$ this is a direct consequence of (\ref{eqn crl}).

\subsection{The Gauss map of a Lorentzian surface in $\R^{2,2}$}
Let $M$ be an oriented Lorentzian surface in $\R^{2,2}.$ We consider its Gauss map 
\begin{eqnarray*}
G:\hspace{.5cm}M&\rightarrow& \Q\\
x&\mapsto& u_1\cdot u_2
\end{eqnarray*}
where, at $x\in M,$ $(u_1,u_2)$ is a positively oriented orthogonal basis of $T_xM$ such that $|u_1|^2=-|u_2|^2=-1.$ The pull-back by the Gauss map of the area form $\omega_{\Q}$ defined in (\ref{def omega Q}) is given by the following proposition:
\begin{prop}\label{pullback_gauss}
We have \[G^*\omega_{\Q}=(K+\sigma K_N)\ \omega_M,\] where $\omega_M$ is the area form, $K$ is the Gauss curvature and $K_N$ is the normal curvature of $M.$  In particular, assuming moreover that 
\begin{equation}
dG_x:\hspace{.5cm}T_xM\to T_{G(x)}\Q 
\end{equation} 
is one-to-one at some point $x\in M,$ then $K=K_N=0$ at $x$ if and only if the linear space $dG_x(T_xM)$ is some $\A$-line in $T_{G(x)}\Q,$ i.e.  
\begin{equation}\label{A_linea}
dG_x(T_xM)=\{a\ U:\ a\in \A \}
\end{equation}
where $U$ is some vector belonging to $T_{G(x)}\Q\subset \mathbb{H}_0.$
\end{prop}
\begin{proof}The first part of the proposition may be obtained by a direct computation, exactly as in \cite[Proposition 6.3]{bayard1}; see also \cite[Proposition 3.1]{BPS} for a similar statement. The second part of the proposition is a consequence of Lemma \ref{lin_indep} in the appendix at the end of the paper.
\end{proof}

As a consequence of Proposition \ref{pullback_gauss}, if $K=K_N=0$ and $G:M\to \Q$ is a regular map, there is a unique Lorentz structure $\sigma$ on $M$ such that  \begin{equation}\label{lorentz}
dG_x(\sigma\ u)=\sigma\ dG_x(u)
\end{equation}
for all $x\in M$ and all $u\in T_xM.$ Indeed, (\ref{A_linea}) implies that $dG_x(T_xM)$ is stable by multiplication by $\sigma,$ and we may define
$$\sigma\ u:=dG_x^{-1}\left(\sigma\ dG_x(u)\right).$$

\subsection{The invariant $\Delta$ of a Lorentzian surface in $\R^{2,2}$}\label{section delta}
If the Gauss map of $M$ is viewed as a map $G:M\rightarrow \Lambda^2\R^{2,2},$ we define
$$\delta(u):=\frac{1}{2}dG_x(u)\wedge dG_x(u)\hspace{.5cm}\in\ \Lambda^4\R^{2,2}$$
for all $u\in T_xM;$ using the canonical volume element $e_0\wedge e_1\wedge e_2\wedge e_3,$ we may identify $\Lambda^4\R^{2,2}$ to $\R$ and thus consider $\delta$ as a quadratic form on $T_xM;$ its determinant with respect to the natural metric on $M$
$$\Delta:=\det{}_g\delta$$
is a second order invariant of the surface; it is positive if and only if the surface admits two distinct asymptotic directions at every point (since an asymptotic direction is by definition a vector vanishing $\delta$ and the sign of $\Delta$ is the opposite of the discriminant of $\delta$), see \cite{BPS}. This invariant was introduced for surfaces in 4-dimensional euclidian space in \cite{Little}.

\subsection{Local description of the flat Lorentzian surfaces with flat normal bundle}
In this section we suppose that $M$ is simply connected and that the bundles $TM$ and $E$ are flat ($K=K_N=0$). We recall that the bundle $\Sigma:=\Sigma E\otimes\Sigma M$ is associated to the principal bundle $\tilde{Q}$ and to the representation $\rho$ of the structure group $Spin'(1,1)\times Spin''(1,1)$ in $\mathbb{H}_0$ given by (\ref{repres_grupo_estructura}). Since the curvatures $K$ and $K_N$ are zero, the spinorial connection on the bundle $\tilde{Q}$ is flat, and $\tilde{Q}$ admits a parallel local section $\tilde{s};$ since $M$ is simply connected, the section $\tilde{s}$ is in fact globally defined. We consider $\varphi\in\Gamma(\Sigma)$ a solution of (\ref{eqn dirac th}) and $g=[\varphi]:M\to Spin(2,2)\subset\mathbb{H}_0$ the coordinates of $\varphi$ in $\tilde{s}:$ \begin{equation*}
\varphi=[\tilde{s},g]\hspace{.5cm} \in \ \Sigma=\tilde{Q}\times\mathbb{H}_0/\rho.
\end{equation*}
Note that, by Theorem \ref{thm_prin}, $\varphi$ also satisfies 
\begin{equation}\label{killing local description}
\nabla_X\varphi=\eta(X)\cdot\varphi
\end{equation} 
for all $X\in TM,$ where 
\begin{equation}\label{special form eta}
\eta(X)=-\frac{1}{2}\sum_{j=2,3}\epsilon_j e_j\cdot B(X,e_j)
\end{equation} 
for some bilinear map $B:TM\times TM \to E.$ In the following, we will denote by $(e_0,e_1)$ and $(e_2,e_3)$ the parallel, orthonormal and positively oriented frames, respectively normal, and tangent to $M,$ corresponding to $\tilde{s},$ i.e. such that $\pi(\tilde{s})=(e_0,e_1,e_2,e_3)$ where $\pi:\tilde{Q}\to Q_E\times Q_M$ is the natural projection. We moreover assume that the Gauss map $G$ of the immersion defined by $\varphi$ is regular, and consider the Lorentz structure $\sigma$ induced on $M$ by $G,$ defined by (\ref{lorentz}). 

We now show that $g$ is in fact a conformal map admitting a special parametri\-zation, and that, in such a special parametrization, $g$ depends on a single conformal map $\psi:\U\subset\A\to\A$ (see Section \ref{appendix lorentz} for the notion of conformal map on a Lorentz surface). To establish this result,  we will first need some preliminary lemmas; since they are analogous to lemmas given in \cite{bayard1}, we only give very brief indications of their proofs, and refer to this paper for details. 

\begin{lem}
Let $g=[\varphi]:M\to Spin(2,2)\subset \mathbb{H}_0$ represent $\varphi$ in some local section of $\tilde{Q}.$ The Gauss map of the immersion defined by $\varphi$ is given by 
\begin{eqnarray}\label{gauss_map_spinorial}
G:\hspace{.5cm} M & \longrightarrow & \Q\hspace{.5cm}\subset \Im m\ \mathbb{H}_0 \\
x & \longmapsto & i\ g^{-1}Ig. \notag
\end{eqnarray}
\end{lem}
\begin{proof}
This is the identity 
$$G=\langle\langle e_2\cdot\varphi,\varphi\rangle\rangle\langle\langle e_3\cdot\varphi,\varphi\rangle\rangle$$ 
written in a section of $\tilde{Q}$ above $(e_2,e_3).$
\end{proof}
\begin{lem}\label{dg g inverse}
Denoting by $[\eta]\in\Omega^1(M,\mathbb{H}_0)$ the $1$-form which represents $\eta$ in $\tilde{s},$ we have 
\begin{equation}\label{eta_coordenadas}
[\eta]=dg\ g^{-1}=\eta_1J+i\eta_2K, 
\end{equation} 
where $\eta_1$ and $\eta_2$ are $1$-forms on $M$ with values in $\A.$
\end{lem}
\begin{proof}
This is (\ref{killing local description}) in the parallel frame $\tilde{s},$ taking into account the special form (\ref{special form eta}) of $\eta$ for the last equality.
\end{proof}
\begin{lem} 
The 1-form
\begin{equation}\label{def eta tilde}
\tilde{\eta}:=\sigma i\ \laa \eta\cdot \varphi,\varphi\raa
\end{equation} 
satisfies $\tilde{\eta}=-\frac{1}{2}G^{-1}dG=-g^{-1}dg.$
\end{lem}
\begin{proof}
Identity (\ref{def eta tilde}) in $\tilde{s}$ together with (\ref{eta_coordenadas}) imply that $\tilde{\eta}= -g^{-1}dg.$ The other identity is an easy consequence of (\ref{gauss_map_spinorial}).
\end{proof}
The properties (\ref{gauss_map_spinorial}) and (\ref{eta_coordenadas}) may be rewritten as follows: 
\begin{lem}\label{distribucion}
Consider the projection 
\begin{align*}
p:\hspace{.5cm}Spin(2,2)\hspace{.2cm}\subset \mathbb{H}_0 &\longrightarrow \Q\hspace{.2cm}\subset \Im m\ \mathbb{H}_0\\
g & \longmapsto i\ g^{-1}Ig
\end{align*}as a $S_{\A}^1$-principal bundle, where the action of $S_{\A}^1$ on $Spin(2,2)$ is given by the multiplication on the left. It is equipped with the horizontal distribution given at every $g\in Spin(2,2)$ by 
\begin{equation} \mathcal{H}_g:=dR_{g}(\A J\oplus i\A K)\hspace{.3cm}\subset\ T_gSpin(2,2), 
\end{equation} 
where $R_{g}$ stands for the right-multiplication by $g$ on $Spin(2,2).$ The distribution $(\mathcal{H}_g)_{g\in Spin(2,2)}$ is $H$-orthogonal to the fibers of $p,$ and, for all $g\in Spin(2,2),$ $dp:\mathcal{H}_g\to T_{p(g)}\Q$ is an isomorphism which preserves $\sigma$ and such that
\begin{equation}\label{dist_formula} 
H(dp(u),dp(u))=-4H(u,u)
\end{equation} 
for all $u\in\mathcal{H}_g.$ With these notations, we have 
\begin{equation}\label{gauss_conforme}
G=p\circ g,
\end{equation} 
and the map $g:M\to Spin(2,2)$ appears to be a horizontal lift to $Spin(2,2)$ of the Gauss map $G:M\to \Q.$
\end{lem}
\begin{remar}
The fibration described in the lemma above generalises the Lorentzian Hopf fibration of pseudo-spheres studied in \cite{leon}. See also \cite[Lemma 6.6]{bayard1} for a similar result in 4-dimensional Minkowski space.
\end{remar}
To proceed further, we need to assume that the invariant $\Delta$ does not vanish; we first suppose $\Delta>0,$ and only mention at the end of the section, and without proof, the similar results concerning the case $\Delta<0$ (see also Remark \ref{rmk missing cases} below, where we recall the results obtained in \cite{BPS} concerning the case $\Delta=0$).
\begin{thm}\label{nueva_carta}
Additionally to the assumptions given at the beginning of the section, we suppose that $\Delta$ is positive on $M;$ we then have:
\\
\\1- the map $g:M\to Spin(2,2)\subset\mathbb{H}_0$ is a conformal map, and, at each point of $M,$ there is a local chart $a:\U\subset\A \to M,$ unique up to the action of 
\begin{equation*}
\mathbb{G}:=\{ a\longmapsto \pm a+b:\ b\in\A \},
\end{equation*} 
which is compatible with the orientation of $M$ and such that $g:\U\subset\A\to Spin(2,2)$ satisfies
\begin{equation}\label{g normalized}
H(g',g')\equiv \pm1;
\end{equation}
2- there exists a conformal map $\psi:\U\subset\A\to \A$ such that 
\begin{equation}\label{nueva_carta_2}
g'g^{-1}=\cosh\psi J+i\sinh\psi K\hspace{0.2in}\text{or}\hspace{0.2in}g'g^{-1}=\sinh\psi J+i\cosh\psi K,
\end{equation}
where $a:\U\subset\A \to M$ is a chart defined in 1-.
\end{thm}

\begin{proof}
Let $a:\U\subset\A \to M$ be a chart given by the Lorentz structure induced by $G$ and compatible with the orientation of $M$ (see Section \ref{appendix lorentz}). By Lemma \ref{distribucion}, $g:\U\to Spin(2,2)$ is a conformal map (since so are $G$ and $p$ in (\ref{gauss_conforme})). We consider $g':\U\rightarrow\mathbb{H}_0$ such that $dg=g'da$ (Section \ref{appendix lorentz}). If $\mu:\A\to\A$ is a conformal map, we have 
$$H((g\circ\mu)',(g\circ\mu)')=\mu'^2H(g',g').$$
 We observe that we may find $\mu$ such that 
 \begin{equation}\label{equation mu prime}
 \mu'^2H(g',g')=\pm1.
 \end{equation} Indeed, since $g$ is a conformal map, 
\begin{equation}\label{Hgpuv}
H(g',g')=\frac{1+\sigma}{2}h_1(s)+\frac{1-\sigma}{2}h_2(t) 
\end{equation}
for some functions $h_1$ and $h_2,$ where $s$ and $t\in\R$ are such that $a=\frac{1 +\sigma}{2}s+\frac{1 -\sigma}{2}t$ (see Section \ref{appendix lorentz}); we observe that $\Delta>0$ is equivalent to  $h_1(s)h_2(t)>0:$ by (\ref{dist_formula})-(\ref{gauss_conforme}),
$$H(dG,dG)=-4H(g',g')da^2=-2\left((h_1ds^2+h_2dt^2)+\sigma (h_1ds^2-h_2dt^2)\right);$$
since
$$H(dG,dG)=\langle dG,dG\rangle-\sigma\ dG\wedge dG$$
(see Appendix \ref{appendix H}), we deduce that
$$\delta:=\frac{1}{2}dG\wedge dG=h_1ds^2-h_2dt^2$$
and thus that the discriminant of $\delta$ has the sign of $-h_1h_2;$ the result follows since this sign is also the opposite of $\Delta$ (see Section \ref{section delta}). Setting
\begin{equation*}
\mu'=\frac{1+\sigma}{2}\frac{1}{\sqrt{|h_1|}}+\frac{1-\sigma}{2}\frac{1}{\sqrt{|h_2|}},
\end{equation*}
we have by (\ref{Hgpuv})
\begin{equation*}
\mu'^2H(g',g')= \frac{1+\sigma}{2}sign(h_1)+\frac{1-\sigma}{2}sign(h_2)=sign(h_1),
\end{equation*}
where $sign(h_1)$ is $+1$ if $h_1>0$ and is $-1$ if $h_1<0.$ We then define 
\begin{equation}\label{integracion}
\mu=\frac{1+\sigma}{2}\int_{s_0}^s \frac{1}{\sqrt{|h_1|}}ds+\frac{1-\sigma}{2}\int_{t_0}^{t}\frac{1}{\sqrt{|h_2|}}dt.
\end{equation} 
$\mu$ is clearly a diffeomorphism, and, considering $g\circ \mu$ instead of $g,$ we get a solution of (\ref{g normalized}). Since all the solutions of (\ref{equation mu prime}) preserving orientation are of the form $\pm \mu+b,$ $b\in\A,$ we also obtain the uniqueness of a solution up to the group $\mathbb{G}.$ We now prove the last claim of the theorem. Writing 
\begin{equation*}
g=\frac{1+\sigma}{2}g_1+\frac{1-\sigma}{2}g_2
\end{equation*} 
with $g_1=g_1(s)$ and $g_2=g_2(t)$ belonging to $\R\e\oplus i\R I\oplus\R J\oplus i\R K$ ($g$ is a conformal map) we get \begin{equation*}g'g^{-1}=\frac{1+\sigma}{2}g_1'g_1^{-1}+\frac{1-\sigma}{2}g_2'g_2^{-1}, \end{equation*}with 
$$H(g_1'g_1^{-1},g_1'g_1^{-1})=H(g_2'g_2^{-1},g_2'g_2^{-1})=\pm1.$$ 
Since $g_1'g_1^{-1}$ and $g_2'g_2^{-1}$ belong to $\R J\oplus i\R K$ (Lemma \ref{dg g inverse}), we deduce that 
\begin{equation*}g_1'g_1^{-1}=\cosh\psi_1 J+i\sinh\psi_1K\hspace{0.1in}\text{and}\hspace{0.1in}g_2'g_2^{-1}=\cosh\psi_2 J+i\sinh\psi_2K \end{equation*}or \begin{equation*}g_1'g_1^{-1}=\sinh\psi_1 J+i\cosh\psi_1K \hspace{0.1in}\text{and}\hspace{0.1in} g_2'g_2^{-1}=\sinh\psi_2 J+i\cosh\psi_2K, \end{equation*} 
for $\psi_1=\psi_1(s)$ and $\psi_2=\psi_2(t)\in\R.$ The function 
\begin{equation*}
\psi:=\frac{1+\sigma}{2}\psi_1(s)+\frac{1-\sigma}{2}\psi_2(t)
\end{equation*} 
satisfies (\ref{nueva_carta_2}). 
\end{proof}

We now study the metric of the surface in the special chart $a=u+\sigma v$ adapted to $g,$ given by Theorem \ref{nueva_carta}. We recall that $(e_0,e_1)$ and $(e_2,e_3)$ are the parallel, orthonormal and positively oriented frames, respectively normal, and tangent to $M,$ corresponding to $\tilde{s}.$ Let us write \begin{equation*}
\vec{H}=H_0e_0+H_1e_1.
\end{equation*}
We also consider the tangent lightlike vectors 
$$N_1:=\frac{e_2+e_3}{\sqrt{2}}\hspace{.5cm}\mbox{and}\hspace{.5cm}N_2:=\frac{e_3-e_2}{\sqrt{2}};$$ 
they are such that $\la N_1,N_2\ra=1.$ Finally, we consider the conformal map $\psi:\U\subset\A\to\A$ defined in Theorem \ref{nueva_carta} above and write $\psi\ =\theta_1 + \sigma \theta_2,$ with $\theta_1,\theta_2$ real-valued functions.
\begin{lem}\label{caso_1}We have 
\begin{equation}\label{marco_caso_1} 
N_1=\pm\frac{e^{\theta_1}}{\sqrt{2}}\left( \frac{1}{\lambda}\partial_u+\frac{1}{\mu}\partial_v\right)\hspace{.5cm} \mbox{and}\hspace{.5cm}
N_2=\frac{e^{-\theta_1}}{\sqrt{2}}\left(\frac{1}{\lambda}\partial_u-\frac{1}{\mu}\partial_v\right)
\end{equation}
where $\lambda,\mu\in\R^*$ satisfy
\begin{equation}\label{coordenadas_vector_normal}
\begin{pmatrix}
\frac{1}{\mu} \\
\frac{1}{\lambda}
\end{pmatrix}=-\begin{pmatrix}
\cosh\theta_2 & \sinh\theta_2 \\
\sinh\theta_2 & \cosh\theta_2
\end{pmatrix}\begin{pmatrix}
H_0 \\ H_1
\end{pmatrix}. 
\end{equation}
\end{lem}
\begin{proof}
In the chart $a:\U\subset\A \to M$ introduced above, $e_2,e_3$ are represented by two functions $\underline{e_2},\underline{e_3}:\U\subset\A \to \A.$ In $\tilde{s},$ the Dirac equation (\ref{eqn dirac th}) reads 
\begin{equation*}
-[e_2]\widehat{[\nabla_{e_2}\varphi]}+[e_3]\widehat{[\nabla_{e_3}\varphi]}= [\vec{H}]\widehat{[\varphi]},
\end{equation*}
that is, recalling Section \ref{section notation},
\begin{equation*}
Jdg(e_2)+iKdg(e_3)=(\sigma H_0\e+iH_1I)g;
\end{equation*}
since $dg(e_2)g^{-1}=g'g^{-1}\underline{e_2}$ and $dg(e_3)g^{-1}=g'g^{-1}\underline{e_3}$ and using the first or the second identity in (\ref{nueva_carta_2}), this may be written 
\begin{equation*}
-\begin{pmatrix}\cosh\psi & \sinh\psi\\ \sinh\psi & \cosh\psi \end{pmatrix}\begin{pmatrix}\sigma H_0 \\ H_1\end{pmatrix}=\begin{pmatrix}\underline{e_2}\\ \underline{e_3}\end{pmatrix}
\hspace{0.1in}\text{or}\hspace{0.1in}
\begin{pmatrix}\sinh\psi & \cosh\psi\\ \cosh\psi & \sinh\psi \end{pmatrix}\begin{pmatrix}\sigma H_0 \\ H_1\end{pmatrix}=\begin{pmatrix}\underline{e_2}\\ \underline{e_3}\end{pmatrix}.
\end{equation*} 
Setting $c:=-H_0\sinh\theta_2-H_1\cosh\theta_2$ and $d:=-H_0\cosh\theta_2-H_1\sinh\theta_2,$ these identities read
\begin{equation*}
\begin{cases}
\underline{e_2}=c\sinh\theta_1+\sigma d\cosh\theta_1\\
\underline{e_3}=c\cosh\theta_1+\sigma d\sinh\theta_1
\end{cases}
\hspace{0.1in}\text{or}\hspace{0.1in}
\begin{cases}
\underline{e_2}=-c\cosh\theta_1-\sigma d\sinh\theta_1 \\
\underline{e_3}=-c\sinh\theta_1-\sigma d\cosh\theta_1
\end{cases}
\end{equation*}
(recall (\ref{cos_sen_num_lorn})). Since $\underline{e_2}$ and $\underline{e_3}$ represent the independent vectors $e_2,e_3,$ we have $cd\neq0;$ setting $\lambda=\frac{1}{c}$ and $\mu=\frac{1}{d},$ we finally easily get (\ref{marco_caso_1}) and (\ref{coordenadas_vector_normal}).  
\end{proof}
\begin{prop}\label{prop metric}
In the chart $a=u+\sigma v$ of Theorem \ref{nueva_carta}, the metric reads 
\begin{equation}\label{metrica_caso_1}
\pm(\lambda^2du^2-\mu^2dv^2);
\end{equation}
moreover, $\lambda$ and $\mu$ are solutions of the hyperbolic system
\begin{equation}\label{sistema_caso_1}
\begin{cases}
\partial_u\mu=-\lambda\ \partial_u\theta_2\\
\partial_v\lambda=-\mu\ \partial_v\theta_2.
\end{cases}
\end{equation}
\end{prop}
\begin{proof}
Let $A$ be the matrix of the metric in the basis $(\partial_u,\partial_v).$ If we denote by 
\begin{equation*}
P=\frac{1}{\sqrt{2}}\begin{pmatrix}\pm\frac{e^{\theta_1}}{\lambda} & \frac{e^{-\theta_1}}{\lambda}\\ \pm\frac{e^{\theta_1}}{\mu} & -\frac{e^{-\theta_1}}{\mu} \end{pmatrix}
\end{equation*} 
the matrix representing the vectors $N_1,N_2$ in $(\partial_u,\partial_v)$  (see (\ref{marco_caso_1})), then, since $|N_1|^2=|N_2|^2=0$ and $\la N_1,N_2\ra=1,$ we have  
$$\begin{pmatrix}0 & 1 \\ 1 & 0\end{pmatrix}=P^tAP;$$ 
thus 
\begin{equation*}
A=\pm \begin{pmatrix}\lambda^2 & 0 \\ 0 & -\mu^2 \end{pmatrix},
\end{equation*} 
which is (\ref{metrica_caso_1}). We then compute the Christoffel symbols of this metric using the Christoffel formulas, and easily get 
\begin{equation*}
\Gamma_{uu}^u=\frac{1}{\lambda}\partial_u\lambda,\hspace{0.2in} \Gamma_{vu}^u=\frac{1}{\lambda}\partial_v\lambda,\hspace{0.2in} \Gamma_{uv}^v=\frac{1}{\mu}\partial_u\mu,\hspace{0.2in} \Gamma_{vv}^v=\frac{1}{\mu}\partial_v\mu 
\end{equation*} 
and
\begin{equation*}
\Gamma_{uu}^v=\frac{\lambda}{\mu^2}\partial_v\lambda,\hspace{0.2in} \Gamma_{vv}^u=\frac{\mu}{\lambda^2}\partial_u\mu. 
\end{equation*} 
Writing finally that $(N_1,N_2),$ given by (\ref{marco_caso_1}), is parallel with respect to the metric (\ref{metrica_caso_1}) (since so is $(e_2,e_3)$), we easily get (\ref{sistema_caso_1}).
\end{proof}
We now state the main result of the section:
\begin{thm}\label{descripcion_1}
Let $\psi:\U\subset\A\to \A$ be a conformal map, and $\theta_1,\theta_2:\U\to\R$ be such that $\psi=\theta_1+\sigma\theta_2;$ suppose that $\lambda$ and $\mu$ are solutions of (\ref{sistema_caso_1}) such that $\lambda\mu\neq0,$ and define \begin{equation}\label{marco1_teorema}
\underline{N_1}=\pm\frac{e^{\theta_1}}{\sqrt{2}}\left( \frac{1}{\lambda}+\sigma\frac{1}{\mu}\right)\hspace{0.2in}\text{and}\hspace{0.2in}
\underline{N_2}=\frac{e^{-\theta_1}}{\sqrt{2}}\left(\frac{1}{\lambda}-\sigma\frac{1}{\mu}\right).
\end{equation}Then, if $g:\U\to Spin(2,2)\subset\mathbb{H}_0$ is a conformal map solving \begin{equation}\label{gp theorem}
g'g^{-1}=\cosh\psi J+i\sinh\psi K\hspace{0.2in}\text{or}\hspace{0.2in} g'g^{-1}=\sinh\psi J+i\cosh\psi K,
\end{equation}
and if we set 
\begin{equation}\label{inmersion1_teorema}
\xi:=ig^{-1}\left[\frac{w_2-w_1}{\sqrt{2}}J+\frac{w_2+w_1}{\sqrt{2}}iK\right]\hat{g}
\end{equation} where $w_1,w_2:T\U\to\R$ are the dual forms of $\underline{N_1},\underline{N_2}\in\Gamma(T\U),$ the function $F=\int\xi$ defines a Lorentzian immersion $\U\to\R^{2,2}$ with $K=K_N=0$ and $\Delta>0.$ Reciprocally, the Lorentzian immersions of $M$ into $\R^{2,2}$ such that $K=K_N=0,$ $\Delta>0$ and with regular Gauss map are locally of this form.
\end{thm}
\begin{proof}We first prove the direct statement. We consider the metric on $\U$ such that the vectors $N_1\simeq \underline{N_1}, N_2\simeq \underline{N_2}\in \Gamma(T\U)$ defined by (\ref{marco1_teorema}) form a frame of lightlike vectors of $T\U$ such that $\la N_1,N_2\ra=1:$ this is the metric (\ref{metrica_caso_1}). Since $(\lambda,\mu)$ is a solution of (\ref{sistema_caso_1}), the frame $(N_1,N_2)$ is parallel, and the metric on $\U$ is flat. We also consider the trivial bundle $E=\R^{1,1}\times\U$ with its trivial metric and its trivial connection: the canonical basis $(e_0,e_1)$ of $\R^{1,1}$ defines orthonormal and parallel sections of $E.$ We moreover define $e_2:=\frac{N_1-N_2}{\sqrt{2}},\ e_3:=\frac{N_1+N_2}{\sqrt{2}},$ parallel and orthogonal frame with $\la e_2,e_2\ra=-1$ and $\la e_3,e_3\ra=1.$ We write $s=(e_0,e_1,e_2,e_3)\in\Q=(SO(1,1)\times SO(1,1))\times\U,$ and fix $\tilde{s}\in\tilde{\Q}=S^1_{\A}\times\U$ such that $\pi(\tilde{s})=s,$ where $\pi:\tilde{Q}\to \Q$ is the natural double covering. We then consider $\varphi\in\Sigma=\tilde{\Q}\times\mathbb{H}_0/\rho$ such that $[\varphi]=g$ in $\tilde{s}.$ By construction (equations (\ref{gp theorem})), $\varphi$ is a solution of the Dirac equation $D\varphi=\vec{H}\cdot\varphi,$ where $\vec{H}=H_0e_0+H_1e_1$ is defined by (\ref{coordenadas_vector_normal}). Moreover, the form defined by (\ref{inmersion1_teorema}) is such that $\xi=\la\la X\cdot\varphi,\varphi\ra\ra;$ this is thus a closed 1-form, and $F=\int\xi$ is an isometric immersion of $M$ into $\R^{2,2}$ whose normal bundle identifies to $E.$ Thus it is a flat immersion in $\R^{2,2},$ with flat normal bundle; moreover $\Delta>0,$ as it is easily seen using the criterion in the proof of Theorem \ref{nueva_carta} ($H(g',g')=\pm 1$ by (\ref{gp theorem}), that is $h_1=h_2=\pm 1$ in (\ref{Hgpuv})). 

Reciprocally, if $F:M\to \R^{2,2}$ is the immersion of a flat Lorentzian surface with flat normal bundle, $\Delta>0,$ and regular Gauss map, we have \begin{equation}
F=\int\xi,\hspace{0.2in}\text{with}\hspace{0.2in}\xi(X)=\la\la X\cdot\varphi,\varphi\ra\ra, 
\end{equation} 
where $\varphi$ is the restriction to $M$ of the constant spinor field $\sigma \e$ of $\R^{2,2}.$ In a parallel frame $\tilde{s},$ we have $\varphi=[\tilde{s},g],$ where $g:M\to Spin(2,2)\subset\mathbb{H}_0$ is an horizontal and conformal map (Lemma \ref{distribucion} and Theorem \ref{nueva_carta}). In a chart compatible with the Lorentz structure induced by the Gauss map and adapted to $g$ (Theorem \ref{nueva_carta}), $\xi$ is of the form (\ref{inmersion1_teorema}) where $(w_1,w_2)$ is the dual basis of the basis defined by (\ref{marco1_teorema}) and where in this last expression $(\lambda,\mu)$ are solutions of (\ref{sistema_caso_1}).
\end{proof}
\begin{coro}\label{4 functions}
A flat Lorentzian surface with flat normal bundle, regular Gauss map and such that $\Delta>0$ locally depends on 4 real functions of one real variable.
\end{coro}
\begin{proof}
We first note that the function $\psi$ in Theorem \ref{descripcion_1} depends on two functions of one variable: since $\psi:\A\to\A$ is a conformal map, writing  
\begin{equation*}
\psi:=\frac{1+\sigma}{2}\psi_1+\frac{1-\sigma}{2}\psi_2
\end{equation*} we have $\psi_1=\psi_1(s)$ and $\psi_2=\psi_2(t),$ where the coordinates $(s,t)$ are defined in (\ref{def s t}). We then write the system (\ref{sistema_caso_1}) in the coordinates $(s,t)$ and get 
\begin{equation}\label{sistema_caso_1_coord} 
\partial_s \begin{pmatrix}\lambda\\ \mu \end{pmatrix}=\begin{pmatrix}+1 & 0\\0 & -1\end{pmatrix}\partial_t\begin{pmatrix}\lambda\\ \mu \end{pmatrix}-\frac{1}{2} \begin{pmatrix}0 & \psi_1'+\psi_2'\\ \psi_1'-\psi_2' & 0 \end{pmatrix}\begin{pmatrix}\lambda\\ \mu \end{pmatrix};
\end{equation}
this is an hyperbolic system, and we may solve a Cauchy problem: once $\psi_1$ and $\psi_2$ are given, a solution of (\ref{sistema_caso_1_coord}) depends on two functions $\mu(0,t),\lambda(0,t)$ of the variable $t.$ By Theorem \ref{descripcion_1}, the surface depends on $\psi_1(s),\psi_2(t), \mu(0,t)$ and $\lambda(0,t).$ 
\end{proof}
We now briefly describe the case $\Delta<0:$ a theorem similar to Theorem \ref{nueva_carta} holds, replacing (\ref{g normalized}) by $H(g',g')=\pm\sigma$ and (\ref{nueva_carta_2}) by
\begin{equation*} 
g'g^{-1}=\frac{1+\sigma}{2}(\cosh\psi J+i\sinh\psi K)+\frac{1-\sigma}{2}(\sinh\psi J+i\cosh\psi K)
\end{equation*}
or
\begin{equation*}
g'g^{-1}=\frac{1+\sigma}{2}(\sinh\psi J+i\cosh\psi K)+\frac{1-\sigma}{2}(\cosh\psi J+i\sinh\psi K).
\end{equation*} 
Further, formulas (\ref{marco_caso_1}) are replaced by 
\begin{equation}\label{marco_caso_2}
N_1=\pm\frac{e^{\theta_1}}{\sqrt{2}}\left( \frac{\rho}{\nu^2+\rho^2}\partial_s+\frac{\nu}{\nu^2+\rho^2}\partial_t\right)\hspace{.5cm}\mbox{and}\hspace{.5cm}
N_2=\frac{e^{-\theta_1}}{\sqrt{2}}\left(-\frac{\nu}{\nu^2+\rho^2}\partial_s+\frac{\rho}{\nu^2+\rho^2}\partial_t\right) 
\end{equation} 
where $\nu,\rho\in\R$ are such that 
\begin{equation*}
\begin{pmatrix}
\frac{\nu}{\nu^2+\rho^2} \\
\frac{\rho}{\nu^2+\rho^2}
\end{pmatrix}=\frac{1}{2}\begin{pmatrix}
-e^{-\theta_2} & e^{-\theta_2} \\
-e^{\theta_2} & -e^{\theta_2}
\end{pmatrix}\begin{pmatrix}
H_0 \\ H_1
\end{pmatrix}. 
\end{equation*}
Following the line of the proof of Proposition \ref{prop metric}, we get that the metric reads 
\begin{equation*}
\pm 4 \ ( \ \nu\rho(-ds^2+dt^2)+(\rho^2-\nu^2)dsdt \ )
\end{equation*} 
and that $\nu,\rho$ are solutions of the system 
\begin{equation*}
\begin{cases}
\partial_s(\rho^2-\nu^2)+2\partial_t(\nu\rho)= -2(\nu^2+\rho^2)\partial_s\theta_2\\
2\partial_s(\nu\rho)-\partial_t(\rho^2-\nu^2)=-2(\nu^2+\rho^2)\partial_t\theta_2.
\end{cases}
\end{equation*}
Setting $z=s+it,$ $f=\rho-i\nu$ and $F=f^2,$ this system reads $\frac{\partial}{\partial \overline{z}}F=2b|F|$ with $b=-\partial_s\theta_2+i\partial_t\theta_2,$ and thus simplifies to
\begin{equation}\label{eqn general analytic f}
\frac{\partial}{\partial \overline{z}}f=b\overline{f}.
\end{equation}
Solutions of (\ref{eqn general analytic f}) are special cases of \emph{generalised analytic functions} (also called \emph{pseudoanalytic functions}) and are known to be in 1-1 correspondence with analytic functions; see e.g. \cite{Bers}, Section 9. As in Theorem \ref{descripcion_1} and Corollary \ref{4 functions} above, we get the following
\begin{coro}\label{2 functions 1 holomorphic}
A flat Lorentzian surface with flat normal bundle, regular Gauss map and such that $\Delta<0$ locally depends on one analytic function and on two real functions of one real variable.
\end{coro}
\begin{remar}\label{rmk missing cases}
If $\Delta=0,$ then all the four natural invariants $K,K_N,|\vec{H}|^2,\Delta$ are zero. Moreover, if the surface does not belong to any degenerate hyperplane of $\R^{2,2},$ it is umbilic or quasi-umbilic and it has a parametrization of the form $\psi(s,t)=\gamma(s)+tT(s),$ where $\gamma$ is a lightlike curve in $\R^{2,2}$ and $T$ is some lightlike vector field along $\gamma$ such that  $\gamma'(s)$ and $T(s)$ are independent for all values of $s.$ We refer to \cite{BPS} for the proof and more details.
\end{remar}

\appendix
\section{Appendix}
\subsection{The norm $H$ on bivectors}\label{appendix H}
We keep the notation of Section \ref{prelim subsection1}.
\begin{prop}
For all $\xi\in \Im m\ \mathbb{H}_0\simeq\Lambda^2\R^{2,2},$
$$H(\xi,\xi)=\langle\xi,\xi\rangle-\sigma\ \xi\wedge\xi.$$
In this formula, $\langle.,.\rangle$ stands for the natural scalar product on $\Lambda^2\R^{2,2},$ and we use the identification $\Lambda^4\R^{2,2}\simeq\R$ given by the canonical volume element $e_0\wedge e_1\wedge e_2\wedge e_3$ to see the term $\xi\wedge\xi$ as a real number.
\end{prop}
\begin{proof}
This is merely a computation: if $\xi=i\xi_1 I+\xi_2J+i\xi_3K$ belongs to $\Im m\ \mathbb{H}_0,$ writing $\xi_j=u_j+\sigma v_j,$ $u_j,v_j\in\R,$ for $j=1,2,3,$ we get
\begin{eqnarray}
H(\xi,\xi)&=&-\xi_1^2+\xi_2^2-\xi_3^2\nonumber\\
&=&-(u_1^2+v_1^2)+(u_2^2+v_2^2)-(u_3^2+v_3^2)-2\sigma\left(u_1v_1-u_2v_2+u_3v_3\right).\label{H eta xy}
\end{eqnarray}
Using the Clifford map (\ref{aplicliff}), the quaternions $iI,$ $\sigma iI,$ $J,$ $\sigma J,$ $iK,$ $\sigma iK$ represent the bivectors $e_2\wedge e_3,$ $e_0\wedge e_1,$ $e_3\wedge e_1,$ $e_2\wedge e_0,$ $e_1\wedge e_2$ and $e_0\wedge e_3$ respectively, and 
$$\xi\simeq u_1\ e_2\wedge e_3+v_1\ e_0\wedge e_1+u_2\ e_3\wedge e_1+v_2\ e_2\wedge e_0+u_3\ e_1\wedge e_2+v_3\ e_0\wedge e_3.$$
Here $(e_0,e_1,e_2,e_3)$ is the canonical basis of $\R^{2,2}.$ It is then straightforward to verify that the term (\ref{H eta xy}) is $\langle\xi,\xi\rangle-\sigma\ \xi\wedge\xi.$  
\end{proof}
\subsection{Vanishing of the area form on the Grassmannian}
We keep here the notation of Section \ref{section grassmannian}.
\begin{lem}\label{lin_indep}
If $\xi,\xi'\in T_p\Q\ \subset\Im m\ \mathbb{H}_0$ are such that $\omega_{\Q}{}_p(\xi,\xi')=0$ then 
\begin{equation*}
\xi'=\lambda\xi,\hspace{0.2in}\xi=\mu\xi'\hspace{0.2in}\text{or} \hspace{0.2in}\xi+\xi'=\pm\sigma(\xi-\xi')
\end{equation*}
for some $\lambda,\mu\in\A.$ In particular the real vector space generated by $\xi$ and $\xi'$ belongs to a $\A$-line in $T_p\Q.$ 
\end{lem}
\begin{proof}
First, it is easy to see that $\omega_{\Q}{}_p(\xi,\xi')=0$ if and only if $\xi\times\xi'=0.$ If we write 
\begin{equation*}
\xi=\frac{1+\sigma}{2}\xi_1+\frac{1-\sigma}{2}\xi_2\hspace{.5cm}\mbox{and}\hspace{.5cm}\xi'=\frac{1+\sigma}{2}\xi_1'+\frac{1-\sigma}{2}\xi_2',
\end{equation*}
where $\xi_1,\xi_2,\xi_1',\xi_2'$ belong to $i\R I\oplus \R J\oplus i\R K\simeq \R^3,$ then $\xi\times\xi'=0$ if and only if $\xi_1\times\xi_1'=\xi_2\times\xi_2'=0$ where the cross product is here the usual cross product in $\R^3.$ We then assume that $\xi$ and $\xi'$ are not zero (else, the result is trivial), and consider the following cases:
\\\emph{1-} If $\xi_1$ and $\xi_2$ are not zero, then $\xi'_1=\alpha\xi_1$ and $\xi'_2=\beta\xi_2$ for some $\alpha,\beta\in\R;$ setting  $\lambda=\frac{1+\sigma}{2}\alpha+\frac{1-\sigma}{2}\beta$ we have $\xi'=\lambda\xi.$ 
\\\emph{2-} If $\xi_1\neq 0$ and $\xi_2=0,$ then,
\\\emph{a-} assuming $\xi'_1\neq 0$ and $\xi'_2=0,$ we have $\xi'_1=\alpha\xi_1$ for some $\alpha\in\R,$ and thus $\xi'=\lambda\xi$ with $\lambda=\frac{1+\sigma}{2}\alpha;$
\\\emph{b-} assuming $\xi'_1=0$ and $\xi'_2\neq 0,$ we have $\xi+\xi'=\sigma(\xi-\xi')$ by a direct computation. 
\\The other cases are similar. Finally, if $\xi'=\lambda\xi$ or $\xi=\mu\xi',$ the real vector space generated by $\xi$ and $\xi'$ obviously belongs to a $\A-$line in $T_p\mathcal{Q};$ this result also holds if $\xi+\xi'=\pm\sigma(\xi-\xi')$ since this space is also generated by $\xi+\xi'$ and $\xi-\xi'.$
\end{proof}
\noindent\textbf{Acknowledgements:} This work is part of the second author's PhD thesis; he thanks CONACYT for support. The authors thank the anonymous referee for many valuable comments which helped to improve considerably the writing of the paper.

\end{document}